\RequirePackage{fix-cm}
\documentclass[smallextended]{svjour3}       
\smartqed  
\usepackage{graphicx}
\begin{document}

\title{The Solvability Of Magneto-heating Coupling Model
 With Turbulent Convection Zone
And The Flow Fields 
}

\titlerunning{Magneto-heating coupling model}        

\author{Changhui Yao        \and
Yanping Lin \and
         Lixiu Wang      \and
       Xuefan Jia
        }

\authorrunning{C. Yao, L.Wang and X. Jia} 

\institute{  Corresponding author:         Changhui Yao \at
              School of Mathematics and Statistics, Zhengzhou University,450001, China.
               \email{chyao@lsec.cc.ac.cn}  \\
               Yanping Lin\at
               Department of Applied Mathematics, Hong Kong Polytechnic University, Hong Kong.\email{yanping.lin@ployu.edu.hk}.\\
              Lixiu Wang  \at  Beijing Computational Science Research Center, Beijing 100193,
China. \email{lixiuwang@csrc.ac.cn}.\\
               Xuefan Jia \at
              School of Mathematics and Statistics, Zhengzhou University,450001, China.
               \email{chyao@zzu.edu.cn}  \\
}
\date{Received: date / Accepted: date}

\maketitle

\begin{abstract}
In this paper, the magneto-heating coupling model is studied in details, with turbulent convection zone and the flow field involved. Our main work is to analyze the well-posed property of this model with the regularity techniques. For the magnetic field, we consider the space $H_0(curl)\cap H(div_0)$ and for the heat equation, we consider the space $H_0^1(\Omega)$. Then we present the weak formulation of the coupled magneto-heating model and establish the regularity problem. Using Roth's method,  monotone theories of nonlinear operator,  weak convergence theories, we prove that the limits of the solutions from Roth's  method converge to the solutions of the regularity problem with proper initial data. With the help of the spacial regularity technique, we derive the results of the well-posedness of the original problems  when the regular parameter $\epsilon\longrightarrow 0$. Moreover, with additional regularity assumption for both the magnetic field and temperature variable, we prove the uniqueness of the solutions.

\end{abstract}

\section{Introduction}
\label{intro}

It is well known that the manifestation of magnetohydrodynamic dynamo (MHD) processes
can be applied to demonstrate large-scale  magnetic activities~\cite{MR01}.
Assume that the magnetic field ${\bf B}$, the electric field ${\bf E}$ and the electric
current density ${\bf J } $ are governed by the Maxwell's equations and constitutive
relations in the magnetohydrodynamic approximation, that is \cite{MR01},
\begin{eqnarray}
\label{equation:eq-1}
&& \partial_t{\bf B}+\nabla\times{\bf E}=0,\ \ \nabla\cdot {\bf B}=0,\\
\label{equation:eq-2}
&& \nabla\times{\bf B}=\mu {\bf J},\ \ {\bf J}=\sigma({\bf E}+{\bf U\times B}),
\end{eqnarray}
where $\mu$ and $\sigma$ are the magnetic permeability and the electric conductivity,
and ${\bf U}$ is the velocity of the fluid.

Large-scale  magnetic and flow fields activities can also drive small-scale turbulent flows as well as large-scale global circulations in their interiors~\cite{MR02,MR03}. Then it is useful to introduce mean-field dynamo theory~\cite{MR04}, which describes the large-scale behavior of such fields. The magnetic and velocity fields can be divided into mean fields and deviations (called ``fluctuations"), ${\bf B}=\bar{\bf B}+{\bf b}$ and ${\bf U}=\bar{\bf U}+{\bf u}$. The equations (\ref{equation:eq-1})-(\ref{equation:eq-2}) can be averaged by
\begin{eqnarray}
\label{equation:eq-3}
&& \partial_t\bar{\bf B}+\nabla\times\bar{\bf E}=0,\ \ \nabla\cdot \bar{\bf B}=0,\\
\label{equation:eq-4}
&& \nabla\times\bar{\bf B}=\mu \bar{\bf J},\ \ \bar{\bf J}=\sigma(\bar{\bf E}+\bar{\bf U}\times\bar{\bf B}+\mathcal{E}),
\end{eqnarray}
where $\mathcal{E}$ is the mean electromotive force due to fluctuations; it is crucial variable for all mean-field electrodynamics:
$$\mathcal{E}=\overline{\bf u\times b}. $$

In order to discuss $\mathcal{E}$, its mean part $\bar{\bf U}$ and the fluctuations ${\bf u}$ are assumed to be known. Then the fluctuations ${\bf b}$ are determined by
\begin{eqnarray}
\label{equation:eq-5}
\eta \nabla\cdot\nabla {\bf b}+\nabla\times(\bar{\bf U}\times{\bf b}+{\bf G})
-\partial_t{\bf b}&=&-\nabla\times({\bf u}\times\bar{\bf B}),\\
\label{equation:eq-6}
{\bf G}&=&({\bf u}\times{\bf B})-\overline{{\bf u}\times{\bf B}}.
\end{eqnarray}
This equation implies that ${\bf b}$ can be considered as a sum ${\bf b}^{0}
+{\bf b}^{\bar{\bf B}}$, where ${\bf b}^{0} $ is  independent of $\bar{\bf B}$ and ${\bf b}^{\bar{\bf B}}$ is a linear and homogeneous in $\bar{\bf B}$. This in turn leads to
$$ \mathcal{E}=\mathcal{E}^0+\mathcal{E}^{\bar{\bf B}}$$
in which $\mathcal{E}^0 $ is  independent of $\bar{\bf B}$ and $\mathcal{E}^{\bar{\bf B}}$ is a linear and homogeneous in $\bar{\bf B}$.

\par
For simplicity, we assume that there is no mean motion, and ${\bf u}$ corresponds to a homogeneous isotropic turbulence. One can derive the relationship
\begin{eqnarray}
\label{equation:eq-7}
\mathcal{E}=\alpha\bar{\bf B}-\beta \nabla\times \bar{\bf B},
\end{eqnarray}
where the two coefficients, $\alpha$ and $\beta$, are independent of position and are determined by ${\bf u}$, and $\eta=\frac{1}{\mu\sigma}$. The term $\alpha\bar{\bf B}$ describes the $\alpha$-effect.
Substituting (\ref{equation:eq-7})into (\ref{equation:eq-3})-(\ref{equation:eq-4}), one can get
\begin{eqnarray}
\label{equation:eq-8}
\partial_t\bar{\bf B}+\nabla\times((\eta+\beta)\nabla\times\bar{\bf B}&=&
\nabla\times(\alpha\bar{\bf B})+\nabla\times(\bar{\bf U}\times\bar{\bf B}),\\
\label{equation:eq-9}
\nabla\cdot \bar{\bf B}&=&0.
\end{eqnarray}
Here $ \lambda=:\eta+\beta$ is the effective magnetic diffusivity, covering both magnetic diffusion at the microscopic level and the turbulent diffusion, respectively and it is also effected by the temperature. The $\alpha$ term represents the turbulent magnetic helicity. In order to deal with the feedback of the magnetic field on fluid motions (the Lorentz force), we employ a so-called $\alpha$-effect or $\alpha$-quench \cite{MR05} by the form
\begin{eqnarray}
\label{equation:eq-10}
\alpha(\bar{\bf B})=\frac{\alpha_0 f({\bf x},t) }{1+(\hat{R_m})^n|\bar{\bf B}/B_{eq}|^2},
\end{eqnarray}
where $\alpha_0>0$  is constant, $0\leq n\leq 2$, $f({\bf x}, t) $ is a model-oriented function, and the $\hat{R_m}$ dependent quenching expression should be regarded as a simplified steady state expression for the nonlinear dynamo~\cite{MR06}, ${B_{eq}}$ is the equipartition magnetic field and can be assumed as a constant. For the convenience,  here and later, we still denote $\bar{\bf B}$ by ${\bf B}$ and simplify (\ref{equation:eq-8})-(\ref{equation:eq-9}) by the following  form with
 $\theta({\bf x},t)$ denoting the temperature at location ${\bf x}\in \Omega$ and time $t$.
 \begin{eqnarray}
\partial_t{ \bf B}+\nabla\times(\lambda(\theta)\nabla\times {\bf B})
-{\Lambda}\nabla(\nabla\cdot{\bf B})
&=&R_\alpha\nabla\times(\frac{f({\bf x},t){\bf B}}{1+\gamma|{\bf B}|^2})\nonumber\\
\label{equation:eq-11}
&+&\nabla\times({\bf U\times B}),  \hspace{0.1cm}\ \ in \ \ (0,T]\times\Omega,\\
\label{equation:eq-12}
 \nabla \cdot{\bf B}&=&0,\ \ \hspace{2cm}\ in \ \ (0,T]\times\Omega,
\end{eqnarray}
where $\lambda(\theta)$ is  bounded and strictly positive i.e. $0<\lambda_0\leq\lambda\leq \lambda_M< +\infty$, $\gamma$ is a constant parameter, $R_\alpha$ is a dynamo parameter in connection with the generation process of small scale turbulence.
With the boundary condition
\begin{equation}
\label{equation:eq-13}
\lambda(\theta)\nabla\times{\bf B}\times {\bf n}=0, \ \ on \ \ \partial\Omega,
\end{equation}
and the initial data
$$B({\bf x},0)={\bf B}_0({\bf x}).$$

\par
  The local density of
Joule's heat  equation generated by
$${\bf E} \cdot {\bf J}=\sigma(|\nabla\times {\bf B}|^2-\nabla\times {\bf B}\cdot({\bf U\times B})
-R_\alpha\nabla\times {\bf B}\cdot(\frac{f({\bf x},t){\bf B}}{1+\gamma|{\bf B}|^2})).$$
Thus, from Fourier¡¯s law and the conservation of energy \cite{MR07,MR08,MR09}, we see that $\theta({\bf x},t)$ satisfies
\begin{eqnarray}
{\partial_t\theta}-\nabla\cdot(\kappa \nabla\theta)&=&\sigma(\theta)(|\nabla\times {\bf B}|^2-\nabla\times {\bf B}\cdot({\bf U\times B})\nonumber\\
\label{equation:eq-14}
&-&R_\alpha\nabla\times {\bf B}\cdot(\frac{f({\bf x},t){\bf B}}{1+\gamma|{\bf B}|^2})),\ \ \ in\ (0,T]\times\Omega,
\end{eqnarray}
with the initial data and boundary conditions \cite{MR09}
\begin{eqnarray}
\label{equation:eq-15}
&&\theta({\bf x},0)=\theta_0,\ \ in \ \ \Omega,\\
\label{equation;eq-16}
&&\theta=\theta_0, \ \ on\ \ (0,T]\times\Gamma_1,\\
\label{equation:eq-17}
&&-\kappa\frac{\partial\theta}{\partial {\bf n}}=\zeta(\theta^4-\theta_0^4)+
\omega(\theta-\theta_0),\ \ on\ \ (0,T]\times\Gamma_2,
\end{eqnarray}
where $\theta_0\in L^\infty(\Omega\cup\Gamma_1)$ is the background temperature, $\partial\Omega=\Gamma_1\cup\Gamma_2$, $\zeta$ is the heat convection coefficient and $\omega$ the radiation coefficient, $\kappa$ is the thermal conductivity  and other physical constants such as
density and specific heat have been normalized. ${\bf n}$  is the unit outer normal to $\Omega$. $\theta_0$  and $\kappa$  are reasonable to assume that
\begin{equation}
\label{equation:eq-18}
\theta_0\geq \theta_{min}>0, \kappa\geq \kappa_{min}>0,
\end{equation}
where $\theta_{min},\kappa_{min} $are positive constants.
For convenience, we define for the positive temperature function
$$\Psi(\theta)= \zeta\theta^4+
\omega\theta:=(\zeta|\theta|^3+
\omega)\theta.$$
Let $\theta=\xi+\theta_0$, we have
$$\Psi(\theta)=\Psi(\xi+\theta_0),\ \ \Psi(\theta)-\Psi(\theta_0)=\zeta(\theta^4-\theta_0^4)+
\omega(\theta-\theta_0).$$
We also define
\begin{eqnarray*}
&&q(\xi):=\sigma(\theta)=\sigma(\xi+\theta_0),\\
&&\mathcal{K}({\bf B})=(|\nabla\times {\bf B}|^2-\nabla\times {\bf B}\cdot({\bf U\times B})
-R_\alpha\nabla\times {\bf B}\cdot(\frac{f({\bf x},t){\bf B}}{1+\gamma|{\bf B}|^2})),
\end{eqnarray*}
and
$$Q_T=(0,T]\times \Omega.$$

\par
The phenomenon of  magneto-heating has been the main point of interest for many researches. In \cite{MR13}, the authors aim to develop a mathematical model for magnetohydrodynamic flow of
biofluids through a hydrophobic micro-channel with periodically contracting and expanding
walls under the influence of an axially applied electric field, and
different temperature jump factors have also been
used to investigate the thermomechanical interactions at the fluid-solid interface. In \cite{MR14}, the authors aim is to investigate the mixed convection flow of an electrically conducting
and viscous incompressible fluid past an isothermal vertical surface with Joule heating in the presence
of a uniform transverse magnetic field fixed relative to the surface.
In \cite{MR15},they study the coupling of the equations of steady-state magnetohydrodynamics (MHD) with
the heat equation when the buoyancy effects due to temperature differences in the flow
as well as Joule effect and viscous heating are  taken into account, wher the existence results of weak solutions are presented  under certain conditions on the data and  some uniqueness results are derived.
In \cite{MR16}, the authors study a coupled system of Maxwell¡¯s equations with nonlinear heat equation while they employ time discretization based on the
Rothe's method to provide energy estimates for discretized system and prove the existence
of a weak solutions  to this coupled system with controlled Joule heating term.

\par
The most significant differences of our mathematical model compared to models stated in papers mentioned above can
be summed into three points:\\
$\bullet$ The  model  coupled with turbulent convection zone and the flow fields.\\
$\bullet$  The nonlinear term concluding $\alpha$-quench.\\
$\bullet$ The coefficient of magnetic diffusion  is temperature dependent and the temperature field is controlled by mixed nonlinear boundaries.

\par
The outline of the paper is as follows:



\section{Preliminaries}
For any $p\leq 1$, let $L^p(\Omega)$ be the sobolev space with the norm
$$\|p\|_{L^p(\Omega)}=(\int_\Omega|p({\bf x})|^pd{\bf x})^{1/p}.$$
For $p=\infty$, $L^\infty(\Omega)$ denotes the space of essentially bounded functions with the norm
$$\|u\|_{L^\infty(\Omega)}=esssup|u({\bf x})|.$$
For $p=2$, $L^2(\Omega)$ denotes the Hilbert space equipped with the inner product and norm
$$(u,v)=\int_\Omega u({\bf x})v({\bf x})d{\bf x},\ \ \ \|u\|_0:=\|u\|_{L^2(\Omega)}=(u,u)^{1/2}. $$
Define $H^m(\Omega)=\{u\in L^2(\Omega): D^{\bf \varsigma} u\in  L^2(\Omega), |{\bf \varsigma}|\leq m \}$, which is equipped with the following norm and semi-norm
$$\|u\|_{m,\Omega}=(\sum\limits_{|{\bf \varsigma}|\leq m}\|D^{\bf \varsigma}u\|_0^2)^{1/2},\ \ \ |u|_{m,\Omega}=(\sum\limits_{|{\bf \varsigma}|= m}\|D^{\bf \varsigma}u\|_0^2)^{1/2}.$$

The most frequently used spaces in the subsequent analysis
are the following two Sobolev spaces:
\begin{eqnarray*}
&&H(curl,\Omega)=\{{\bf u}\in L^2(\Omega)^3;\nabla\times{\bf u}\in  L^2(\Omega)  \},\\
&&H(div,\Omega)=\{{\bf u}\in L^2(\Omega)^3;\nabla\cdot{\bf u}\in  L^2(\Omega)  \}
\end{eqnarray*}
and their subspaces
\begin{eqnarray*}
&&H_0(curl,\Omega)=\{{\bf u}\in H(curl,\Omega), {\bf u}\times {\bf n}=0, \ on\ \   \partial\Omega \},\\
&&H(div_0,\Omega)=\{{\bf u}\in H(div,\Omega), \nabla\cdot{\bf u}=0\in  \Omega \},
\end{eqnarray*}
which are the equipped with the inner product
$$({\bf u},{\bf v})_{H(curl, \Omega)}=({\bf u},{\bf v})+(\nabla\times {\bf u},\nabla\times {\bf v}),$$
$$({\bf u},{\bf v})_{H(div, \Omega)}=({\bf u},{\bf v})+(\nabla\cdot {\bf u},\nabla\cdot {\bf v}), $$
and the norm
$$ \|{\bf u}\|_{H(curl,\Omega)}^2=\|{\bf u}\|_0^2+\|\nabla\times {\bf u}\|_0^2, \ \ \|{\bf u}\|_{H(div,\Omega)}^2=\|{\bf u}\|_0^2+\|\nabla\cdot {\bf u}\|_0^2$$

To treat the constraint equation $\nabla\cdot {\bf B}=0$, we shall need the following subspace
$$\mathcal{V}=H_0(curl,\Omega)\cap H(div_0,\Omega) $$
with the inner product and norm
$$({\bf u},{\bf v})_{\bf V}= ({\bf u},{\bf v})+(\nabla\times {\bf u},\nabla\times {\bf v})+(\nabla\cdot {\bf u},\nabla\cdot {\bf v}),\ \
\|{\bf u}\|_{\bf V}^2=\|{\bf u}\|_0^2+\|\nabla\times {\bf u}\|_0^2 +\|\nabla\cdot {\bf u}\|_0^2.$$
We also need define the functional space for the radiative and conductive heat equation
$$H^1_0(\Omega)=\{v\in H^1(\Omega),v|_{\Gamma_1}=0   \},$$
$$
\mathcal{Y}=\{v\in H^1_0(\Omega)\cap L^5(\Gamma_2) \}, \|v\|_{\mathcal{Y}}:=\|v\|_1+\|v\|_{L^5(\Gamma_2)},
$$
$$W^{0,4}(curl,\Omega)=\{{\bf u}\in L^2(\Omega)^3, \nabla\times {\bf u}\in  L^4(\Omega)^3  \}.$$

The coupling system (\ref{equation:eq-11})-(\ref{equation:eq-17}) can be
is equivalent to the following variational problem:
 Find ${\bf B}\in L^2(0,T;{\mathcal{V}}) $ and $\xi \in L^2(0,T;\mathcal{Y})$
 such that for any
 $ {\bf\Phi}\in\mathcal{V}, \Upsilon\in\mathcal{Y}\cap L^\infty(\Omega)$
\begin{eqnarray}
(\partial_t{ \bf B},{\Phi})&+&(\lambda(\xi+\theta_0)\nabla\times {\bf B},\nabla\times{\Phi})
+{\Lambda}(\nabla\cdot{\bf B},\nabla\cdot{\Phi})=R_\alpha(\frac{f({\bf x},t){\bf B}}{1+\gamma|{\bf B}|^2}, \nabla\times{\Phi})\nonumber\\
\label{equation:eq-19}
&+&({\bf U\times B},\nabla\times{\Phi}), \ \ \forall {\Phi\in \mathcal{V}},\\
(\partial_t\xi, \Upsilon)&+&(\kappa\nabla\xi, \nabla\Upsilon)+<(\Psi(\xi+\theta_0)-\Psi(\theta_0)),\Upsilon>_{\Gamma_2}\nonumber\\
\label{equation:eq-20}
&=&
(q(\xi)\mathcal{K}({\bf B}),\Upsilon)-(\kappa\nabla\theta_0,\nabla\Upsilon),\forall \Upsilon\in \mathcal{Y}\cap L^\infty(\Omega),
\end{eqnarray}
where $<(\Psi(\xi+\theta_0)-\Psi(\theta_0)),\Upsilon>_{\Gamma_2}
=\int_{\Gamma_2}(\Psi(\xi+\theta_0)-\Psi(\theta_0))\Upsilon ds.$

In this paper, we consider the well-posedness of the coupling system (\ref{equation:eq-19})-(\ref{equation:eq-20})  with the regularity technique.
In order to be convenient for the following proofs, we introduce two nonlinear operators defined by: for a given constant $\tau>0$, let $\mathcal{P}:\mathcal{V}\longrightarrow \mathcal{V}'$ and $\mathcal{L}:\mathcal{Y}\longrightarrow \mathcal{Y}'$ such that
\begin{eqnarray}
\label{equation:eq-21}
<\mathcal{P}{\bf A},{ \Phi}>:&=&\frac{1}{\tau}({ \bf A},{ \Phi})
+(\lambda(\xi+\theta_0)\nabla\times {\bf A},\nabla\times{\Phi})
+{\Lambda}(\nabla\cdot{\bf A},\nabla\cdot{\Phi})\nonumber\\
&-&R_\alpha(\frac{f({\bf x},t){\bf A}}{1+\gamma|{\bf A}|^2}, \nabla\times{\Phi})
-({\bf U\times A},\nabla\times{\Phi}),\forall {\bf A},\Phi\in \mathcal{V},\\
<\mathcal{L}\omega,\Upsilon>:&=&\frac{1}{\tau}(\omega,\Upsilon)+(\kappa\nabla\omega, \nabla\Upsilon)\nonumber\\
\label{equation:eq-22}
&+&<\Psi(\omega+\theta_0)-\Psi(\theta_0),\Upsilon>_{\Gamma_2},
\forall \omega,\Upsilon\in \mathcal{Y}.
\end{eqnarray}

\begin{lemma}
\label{lemma:Lem-1}
There exits a constant $C_1$ dependent of $R_\alpha,\lambda_M, \|f\|_{L^\infty(0,T;L^\infty(\Omega))}$,$\|{\bf U}\|_{L^\infty(0,T;L^\infty(\Omega))}$, $C_2$ dependent of $ \zeta,\omega, \Gamma_2$, $C_3$ dependent of $\kappa$,
and parameters $\tau$ such that
\begin{eqnarray}
\label{equation:eq-23}
&&\|\mathcal{P}{\bf B}\|_{\mathcal{V}'}\leq C_1\|{\bf B}\|_{\mathcal{V}},\ \
\|\mathcal{L}\xi\|_{\mathcal{Y}'}  \leq C_3\|\xi\|_1
+C_2(\sum_{j=1}^4\|\xi\|_{L^j(\Gamma_2)}^j).
\end{eqnarray}

\end{lemma}
\begin{proof}
Noting that $\frac{f({\bf x},t)}{1+\gamma|{\bf B}|^2}\leq 1$, $\lambda(\xi+\theta_0)\leq \lambda_M$ and
using Chaucy-Schwarz inequality, we have
\begin{eqnarray}
&&<\mathcal{P}{\bf B},{ \Phi}>:=\frac{1}{\tau}({ \bf B},{ \Phi})
+(\lambda\nabla\times {\bf B},\nabla\times{\Phi})
+{\Lambda}(\nabla\cdot{\bf B},\nabla\cdot{\Phi})\nonumber\\
&-&R_\alpha(\frac{f({\bf x},t){\bf B}}{1+\gamma|{\bf B}|^2}, \nabla\times{\Phi})
-({\bf U\times B},\nabla\times{\Phi}),\nonumber\\
&\leq& \frac{1}{\tau}\|{ \bf B}\|_0\|\Phi\|_0
+\lambda_M\|\nabla\times {\bf B}\|_0\|\nabla\times{\Phi}\|_0
+{\Lambda}\|\nabla\cdot{\bf B}\|_0\|\nabla\cdot{\Phi}\|_0\nonumber\\
&+&R_\alpha\|f({\bf x},t)\|_{{L^\infty(0,T;L^\infty(\Omega))}}
\|{ \bf B}\|_0\|\nabla\times{\Phi}\|_0\nonumber\\
&+&\|U({\bf x},t)\|_{L^\infty(0,T;L^\infty(\Omega))}
\|{ \bf B}\|_0\|\nabla\times{\Phi}\|_0\nonumber\\
\label{equation:eq-24}
&\leq&C_1\|{\bf B}\|_{\mathcal{V}}\|{\bf \Phi}\|_{\mathcal{V}},
\end{eqnarray}
where $C_1=\max\{\frac{1}{\tau},  \lambda_M,R_\alpha\|f({\bf x},t)\|_{L^\infty(0,T;L^\infty(\Omega))},\|U({\bf x},t)\|_{L^\infty(0,T;L^\infty(\Omega))}\}.$
\par
For the function $\theta_0>0$,  we have
\begin{eqnarray}
&&|\Psi(\xi+\theta_0)-\Psi(\theta_0)|=|
(\zeta|\xi+\theta_0|^3+
\omega)(\xi+\theta_0)-(\zeta|\theta_0|^3+
\omega)\theta_0|\nonumber\\
\label{equation:eq-25}
&&\leq |\xi|(\zeta|\xi+\theta_0|^3+3\zeta\xi^2\theta_0+3\zeta\xi\theta_0^2
+3\zeta\theta_0^3+\omega).
\end{eqnarray}
Then there exists a constant $C_2$ dependent
of $ \zeta,\omega, \Gamma_2$ and $\|\theta_0\|_{L^\infty(\Omega)}$
such that
\begin{eqnarray}
&&|\int_{\Gamma_2}(\Psi(\xi+\theta_0)-\Psi(\theta_0))\Upsilon ds|\nonumber\\
&&\leq \|(\Psi(\xi+\theta_0)-\Psi(\theta_0))\|_{L^2(\Gamma_2)}
\|\Upsilon\|_{L^2(\Gamma_2)}\nonumber\\
\label{equation:eq-26}
&&\leq C_2\|\Upsilon\|_{L^2(\Gamma_2)}\sum_{j=1}^4\|\xi\|_{L^j(\Gamma_2)}^j.
\end{eqnarray}
Therefore, there exists a constant $C_3$ dependent of $\tau$ and $\kappa$ so that the boundness of the nonlinear operator $\mathcal{L}$ can be estimated by
\begin{eqnarray}
<\mathcal{L}\xi,\Upsilon>&\leq & \frac{1}{\tau}\|\xi\|_0\|\Upsilon\|_0+\kappa\|\xi\|_1\|\Upsilon\|_1
+(\sum_{j=1}^4\|\xi\|_{L^j(\Gamma_2)}^j)\|\Upsilon\|_{L^2(\Gamma_2)}\nonumber\\
&\leq& \max(\frac{1}{\tau},\kappa)\|\xi\|_1\|\Upsilon\|_1
+C_2(\sum_{j=1}^4\|\xi\|_{L^j(\Gamma_2)}^j)\|\Upsilon\|_{L^2(\Gamma_2)}\nonumber\\
\label{equation:eq-27}
&\leq& C_3\|\xi\|_1\|\Upsilon\|_1
+C_2(\sum_{j=1}^4\|\xi\|_{L^j(\Gamma_2)}^j)\|\Upsilon\|_{L^2(\Gamma_2)}
\end{eqnarray}

\end{proof}


\begin{lemma}
\label{lemma:Lem-2}
There exist a positive constant $C_4$ depending on $\tau, \kappa,\lambda_0, R_\alpha,
\|{ f}\|_{L^\infty(\Omega)},\|{\bf U}\|_{L^\infty(\Omega)}$
 and $C_5$ depending on
$\tau, \kappa$
  such that
\begin{eqnarray}
\label{equation:eq-28}
<\mathcal{P}{\bf B},{\bf B}>\geq  C_4\|{\bf B}\|_{\mathcal{V}}^2
 \ \ <\mathcal{L}\xi,\xi>\geq C_5\|\xi\|^2_1
 +\frac{\zeta}{8}\|\xi\|^5_{L^5(\Gamma_2)}.
\end{eqnarray}
\end{lemma}

\begin{proof} From Young inequality and $\lambda(\xi+\theta_0)\geq \lambda_0$, we have
\begin{eqnarray}
<\mathcal{P}{\bf B},{\bf B}>&=&\frac{1}{\tau}({ \bf B},{ \bf B})
+(\lambda\nabla\times {\bf B},\nabla\times{\bf B})
+{\Lambda}(\nabla\cdot{\bf B},\nabla\cdot{\bf B})\nonumber\\
&-&R_\alpha(\frac{f({\bf x},t){\bf B}}{1+\gamma|{\bf B}|^2}, \nabla\times{\bf B})
+({\bf U\times \bf B},\nabla\times{\bf B})\nonumber\\
&\geq& \frac{1}{\tau}\|{\bf B}\|_0^2+\lambda_0\|\nabla\times {\bf B}\|_0^2
+{\Lambda}\|\nabla\cdot{\bf B}\|_0^2\nonumber\\
&-&R_\alpha\|f({\bf x},t)\|_{L^\infty(\Omega)}\|{\bf B}\|_0\|\nabla\times{\bf B}\|_0
-\|{\bf U}\|_{L^\infty(\Omega)}\|{\bf B}\|_0\|\nabla\times{\bf B}\|_0\nonumber\\
&\geq & \frac{1}{\tau}\|{\bf B}\|_0^2+\lambda_0\|\nabla\times {\bf B}\|_0^2
+{\Lambda}\|\nabla\cdot{\bf B}\|_0^2
-\frac{R_\alpha\|f\|_{L^\infty(\Omega)}}{4\epsilon_1}\|{\bf B}\|_0^2\nonumber\\
&-&\epsilon_1R_\alpha\|f\|_{L^\infty(\Omega)}\|\nabla\times{\bf B}\|_0^2
-\frac{\|{\bf U}\|_{L^\infty(\Omega)}}{4\epsilon_2}\|{\bf B}\|_0^2
-\epsilon_2\|{\bf U}\|_{L^\infty(\Omega)}\|\nabla\times{\bf B}\|_0^2\nonumber\\
&=&(\frac{1}{\tau}-\frac{R_\alpha\|f\|_{L^\infty(\Omega)}}{4\epsilon_1}
-\frac{\|{\bf U}\|_{L^\infty(\Omega)}}{4\epsilon_2})\|{\bf B}\|_0^2
+{\Lambda}\|\nabla\cdot{\bf B}\|_0^2\nonumber\\
&+&(\lambda_0-\epsilon_1R_\alpha\|f\|_{L^\infty(\Omega)}
-\epsilon_2\|{\bf U}\|_{L^\infty(\Omega)})\|\nabla\times{\bf B}\|_0^2\nonumber\\
\label{equation:eq-29}
&\geq & C_4\|{\bf B}\|_{\mathcal{V}}^2,
\end{eqnarray}
after taking $\epsilon_1,\epsilon_2$ and $\tau$ such that
\begin{eqnarray*}
C_4=\min (\frac{1}{\tau}-\frac{R_\alpha\|f\|_{L^\infty(\Omega)}}{4\epsilon_1}-\frac{\|{\bf U}\|_{L^\infty(\Omega)}}{4\epsilon_2}, \Lambda,\lambda_0-\epsilon_1R_\alpha\|f\|_{L^\infty(\Omega)}-\epsilon_2\|{\bf U}\|_{L^\infty(\Omega)}  ).
\end{eqnarray*}


Now we consider the coercive of the nonlinear operator $\mathcal{L}$.
For the function
$$\Psi(t)=\zeta |t|^3t+\omega t,\ \ \Psi'(t)=4\zeta|t|^3+\omega>0,$$
we know that
$\Psi(t)$ is a monotone function, and
there holds
$$<\Psi(v)-\Psi(w),v-w>|_{\Gamma_2}\geq \frac{\zeta}{8}\|v-w\|^5_{L^5(\Gamma_2)}
+\omega\|v-w\|^2_{L^2(\Gamma_2)}, $$
then we have
$$<\Psi(\xi+\theta_0)-\Psi(\theta_0),\xi>|_{\Gamma_2}
\geq \frac{\zeta}{8}\|\xi\|^5_{L^5(\Gamma_2)}+\omega \|\xi\|^2_{L^2(\Gamma_2)},$$
and
$$ <\Psi(v+\theta_0)-\Psi(w+\theta_0),v-w>|_{\Gamma_2}
\geq \frac{\zeta}{8}\|v-w\|^5_{L^5(\Gamma_2)}+\omega\|v-w\|^2_{L^2(\Gamma_2)}.$$
Therefore, we have
\begin{eqnarray}
<\mathcal{L}\xi,\xi>&=&\frac{1}{\tau}(\xi,\xi)+(\kappa\nabla\xi, \nabla\xi)
+<\Psi(\xi+\theta_0)-\Psi(\theta_0),\xi>_{\Gamma_2}\nonumber\\
&\geq& \frac{1}{\tau}\|\xi\|_0^2+\kappa\|\nabla\xi\|^2_0
+\frac{\zeta}{8}\|\xi\|^5_{L^5(\Gamma_2)}+\omega \|\xi\|_{L^2(\Gamma_2)}^2\nonumber\\
\label{equation:eq-30}
&\geq& C_5\|\xi\|^2_1+\frac{\zeta}{8}\|\xi\|^5_{L^5(\Gamma_2)},
\end{eqnarray}
where $C_5$ is take as the $C_5=\min\{\frac{1}{\tau},\kappa \}$.
\end{proof}

\begin{lemma}
\label{lemma:Lem-31}
For the vector ${\bf A,B}$ and the parameter $\gamma>0$, there holds
\begin{eqnarray*}
|\frac{\bf B}{1+\gamma|{\bf B}|^2}
-\frac{\bf A}{1+\gamma|{\bf A}|^2}|\leq \frac{9}{4}|{\bf B}-{\bf A}|.
\end{eqnarray*}
\end{lemma}
\begin{proof}
By calculating, we have
\begin{eqnarray*}
&&|\frac{\bf B}{1+\gamma|{\bf B}|^2}
-\frac{\bf A}{1+\gamma|{\bf A}|^2}|\\
&&\leq \frac{|{\bf B-A}|}{1+\gamma|{\bf B}|^2}+\frac{\gamma|{\bf A}|(|{\bf A}|-|{\bf B}|)(|{\bf A}|+|{\bf B}|)}{(1+\gamma|{\bf A}|^2)(1+\gamma|{\bf B}|^2)}\\
&&\leq |{\bf B-A}|\frac{(1+2\gamma|{\bf A}|^2+\gamma|{\bf A}||{\bf B}|)}{(1+\gamma|{\bf A}|^2)(1+\gamma|{\bf B}|^2)}.
\end{eqnarray*}
By the symmetry, we have
\begin{eqnarray*}
&&|\frac{\bf A}{1+\gamma|{\bf A}|^2}
-\frac{\bf B}{1+\gamma|{\bf B}|^2}|\\
&&\leq |{\bf B-A}|\frac{(1+2\gamma|{\bf B}|^2+\gamma|{\bf A}||{\bf B}|)}{(1+\gamma|{\bf A}|^2)(1+\gamma|{\bf B}|^2)}.
\end{eqnarray*}
Therefore, we have
\begin{eqnarray*}
&&|\frac{\bf A}{1+\gamma|{\bf A}|^2}
-\frac{\bf B}{1+\gamma|{\bf B}|^2}|\\
&&\leq |{\bf B-A}|\frac{(1+\gamma|{\bf A}|^2+\gamma|{\bf B}|^2+\gamma|{\bf A}||{\bf B}|)}{(1+\gamma|{\bf A}|^2)(1+\gamma|{\bf B}|^2)}\\
&&\leq |{\bf B-A}| \frac{(1+\frac{3}{2}\gamma|{\bf A}|^2)(1+\frac{3}{2}\gamma|{\bf B}|^2)}{(1+\gamma|{\bf A}|^2)(1+\gamma|{\bf B}|^2)}\\
&&\leq \frac{9}{4}|{\bf B-A}|.
\end{eqnarray*}

\end{proof}

\begin{lemma}
\label{lemma:Lem-3}
The operator $\mathcal{P}$ and $\mathcal{L}$ is strictly
 monotone in the sense that
 \begin{eqnarray}\label{equation:eq-31}
<\mathcal{P}{\bf B}-\mathcal{P}{\bf A},{\bf B-A}>
\geq  C_6\|{\bf B-A}\|_{\mathcal{V}}^2,
\end{eqnarray}
where $C_6$ is taken as $\min\{(\frac{1}{\tau}-\frac{R_\alpha\|f\|_{L^\infty(\Omega)}}{4\epsilon_3}
-\frac{\|{\bf U}\|_{L^\infty(\Omega)}}{4\epsilon_4}),  {\Lambda},(\lambda_0-\epsilon_3R_\alpha\|f\|_{L^\infty(\Omega)}
-\epsilon_4\|{\bf U}\|_{L^\infty(\Omega)})\}$,
And
\begin{eqnarray}
\label{equation:eq-32}
<\mathcal{L}v-\mathcal{L}w,v-w>
\geq  C_7\|v-w\|_1^2
+\frac{\zeta}{8}\|v-w\|^5_{L^5(\Gamma_2)},
\end{eqnarray}
where the constant $C_7$ can be taken as $C_7=\min (\tau^{-1},\kappa)$.
\end{lemma}

\begin{proof}From the Young inequality and Lemma 3, we have
\begin{eqnarray}
&&<\mathcal{P}{\bf B}-\mathcal{P}{\bf A},{\bf B-A}>\nonumber\\
&=&\frac{1}{\tau}({\bf B-A},{\bf B-A})
+\lambda(\nabla\times({\bf B-A}), \nabla\times({\bf B-A}))
+{\Lambda}(\nabla\cdot({\bf B-A}),\nabla\cdot({\bf B-A}))\nonumber\\
&-&R_\alpha(\frac{f({\bf x},t)}{1+\gamma|{\bf B}|^2}{\bf B}
-\frac{f({\bf x},t)}{1+\gamma|{\bf A}|^2}{\bf A},\nabla\times({\bf B-A}))
-({\bf U}\times ({\bf B-A}),\nabla\times({\bf B-A}))\nonumber\\
&\geq &\frac{1}{\tau}\|{\bf B-A}\|_0^2+\lambda_0\|\nabla\times ({\bf B-A})\|_0^2
+{\Lambda}\|\nabla\cdot({\bf B-A})\|_0^2
-\frac{9R_\alpha\|f\|_{L^\infty(\Omega)}}{16\epsilon_3}\|{\bf B-A}\|_0^2\nonumber\\
&-&\epsilon_3R_\alpha\|f\|_{L^\infty(\Omega)}\|\nabla\times({\bf B-A})\|_0^2
-\frac{\|{\bf U}\|_{L^\infty(\Omega)}}{4\epsilon_4}\|{\bf {\bf B-A}}\|_0^2
-\epsilon_4\|{\bf U}\|_{L^\infty(\Omega)}\|\nabla\times({\bf B-A})\|_0^2\nonumber\\
&=&(\frac{1}{\tau}-\frac{R_\alpha\|f\|_{L^\infty(\Omega)}}{4\epsilon_3}
-\frac{\|{\bf U}\|_{L^\infty(\Omega)}}{4\epsilon_4})\|{\bf B-A}\|_0^2
+{\Lambda}\|\nabla\cdot({\bf B-A})\|_0^2\nonumber\\
&+&(\lambda-\epsilon_3R_\alpha\|f\|_{L^\infty(\Omega)}
-\epsilon_4\|{\bf U}\|_{L^\infty(\Omega)})\|\nabla\times({\bf B-A})\|_0^2\nonumber\\
\label{equation:eq-33}
&\geq & C_6\|{\bf B-A}\|_{\mathcal{V}}^2,
\end{eqnarray}
where $C_6$ is taken as $\min\{(\frac{1}{\tau}-\frac{9R_\alpha\|f\|_{L^\infty(\Omega)}}{16\epsilon_3}
-\frac{\|{\bf U}\|_{L^\infty(\Omega)}}{4\epsilon_4}),  {\Lambda},(\lambda_0-\epsilon_3R_\alpha\|f\|_{L^\infty(\Omega)}
-\epsilon_4\|{\bf U}\|_{L^\infty(\Omega)})\}$.

\par
Since $\Psi(\cdot)$ is a monotone function, we have
\begin{eqnarray}
&&<\mathcal{L}v-\mathcal{L}v,v-w>\nonumber\\
&=&\frac{1}{\tau}\|v-v\|_0^2
+\kappa\|\nabla (v-w)\|_0^2
+\int_{\Gamma_2}(\Psi(v+\theta_0)-\Psi(w+\theta_0))(v-w)ds\nonumber\\
&\geq&\frac{1}{\tau}\|v-v\|_0^2
+\kappa\|\nabla (v-w)\|_0^2
+\frac{\zeta}{8}\|v-w\|^5_{L^2(\Gamma_2)}+\omega\|v-w\|^2_{L^2(\Gamma_2)}\nonumber\\
\label{equation:eq-34}
&\geq & C_7\|v-w\|_1^2
+\frac{\zeta}{8}\|v-w\|^5_{L^5(\Gamma_2)}+\omega\|v-w\|^2_{L^2(\Gamma_2)},
\end{eqnarray}
where $C_7$ is take as the $C_7=\min (\tau^{-1},\kappa)$.
\end{proof}

\begin{lemma}
\label{lemma:Lem-4}
The nonlinear operator $\mathcal{P}:\mathcal{V}\longrightarrow \mathcal{V}'$  and
$\mathcal{L}:\mathcal{Y}\longrightarrow \mathcal{Y}'$ is hemi-continuous, that is
$$\mathcal{S}(s)=<\mathcal{P}({\bf R}+s{\bf Q}), {\bf \Phi}>,\ \
\mathcal{Z}(s)=<\mathcal{L}(v+su), {w}>$$
is continuous on $s\in [0,1]$,respectively, for any ${
\bf Q,R,\Phi}\in \mathcal{V}, u,v,w\in \mathcal{Y}$.
\end{lemma}
\begin{proof}
For convenience, we denote ${\bf Q}(s)={\bf R}+s{\bf Q}$.
For any $s,s_0\in [0,1]$, we have
\begin{eqnarray}
|\mathcal{S}(s)-\mathcal{S}(s_0)|
&=&<\mathcal{P}({\bf Q}(s))-\mathcal{P}({\bf Q}(s_0)), {\bf \Phi}>\nonumber\\
&=&\frac{1}{\tau}({\bf Q}(s)-{\bf Q}(s_0),{\bf \Phi})
+\lambda(\nabla\times({\bf Q}(s)-{\bf Q}(s_0)),\nabla\times{\bf \Phi})\nonumber\\
&+&\Lambda(\nabla\cdot({\bf Q}(s)-{\bf Q}(s_0)),\nabla\cdot{\bf \Phi})
-({\bf U}\times({\bf Q}(s)-{\bf Q}(s_0)),\nabla\times{\bf \Phi})\nonumber\\
&-&(R_\alpha\frac{f({\bf x},t)}{1+\gamma|{\bf Q}(s)|^2}{\bf Q}(s)
- R_\alpha\frac{f({\bf x},t)}{1+\gamma|{\bf Q}(s_0)|^2}{\bf Q}(s_0),\nabla\times{\bf \Phi})\nonumber\\
&\leq & \frac{1}{\tau}\|{\bf Q}\|_0\|{\bf \Phi}\|_0|s-s_0|
+\lambda_0\|\nabla\times{\bf Q}\|_0\|\nabla\times{\bf \Phi}\|_0|s-s_0 |\nonumber\\
&+&\Lambda\|\nabla\cdot{\bf Q}\|_0\|\nabla\cdot{\bf \Phi}\|_0|s-s_0 |
+\|{\bf U}\times{\bf Q}\|_0\|\nabla\times{\bf \Phi}\|_0|s-s_0 |\nonumber\\
&+&R_\alpha\|f({\bf x},t)\|_{L^\infty(\Omega)}\|{\bf Q}\|_0
\|\nabla\times{\bf \Phi}\|_0|s-s_0 |,
\end{eqnarray}
where we use $ \frac{1}{1+\gamma|{\bf Q}(s)|^2}\leq 1,\
\frac{1}{1+\gamma|{\bf Q}(s_0)|^2}\leq 1$. This shows that
$\mathcal{S}(s)$ is continuous on $ [0,1]$  for any ${\bf Q},{\bf R}\in \mathcal{Y}.$

\par
We also denote $u(t)=v+su, \forall u,v\in \mathcal{Y}, t\in [0,1]$.
Then for any $s,s_0\in [0,1]$, we have
\begin{eqnarray}
&&|\mathcal{Z}(s)-\mathcal{Z}(s_0)|
=<\mathcal{L}u(s)-\mathcal{L}u(s_0), {w}>\nonumber\\
&=&\frac{1}{\tau}(u(s)-u(s_0),w)+\kappa(\nabla(u(s)-u(s_0)),\nabla w)+\int_{\Gamma_2}
(\Psi(u(t))-\Psi(u(t_0)))wds\nonumber\\
&\leq&|s-s_0|(\frac{1}{\tau}\|u\|_0\|w\|_0+\kappa\|\nabla u\|_0\|\nabla w\|_0
+\int_{\Gamma_2}4\zeta((|u|+|v|)^4 +\omega|v|)wds,
\end{eqnarray}
which means $\mathcal{Z}(s)$ is continuous on $ [0,1]$ for any $u,v\in \mathcal{Y}$.

\end{proof}

\section{The Regularized Problem}
We have to notice the test function $\Upsilon\in L^\infty(\Omega)$
in (\ref{equation:eq-20}),which increases the difficulties deeply when analyzing
the well-posedness. In order deal with this problem, The Regularized techniques can be
employed: given the small parameter $0<\epsilon<1$, find
 ${\bf B}\in L^2(0,T;{\mathcal{V}}) $ and $\xi \in L^2(0,T;\mathcal{Y})$
 such that
\begin{eqnarray}
(\partial_t{ \bf B},{\Phi})&+&(\lambda(\xi+\theta_0)\nabla\times {\bf B},\nabla\times{\Phi})
+{\Lambda}(\nabla\cdot{\bf B},\nabla\cdot{\Phi})=R_\alpha(\frac{f({\bf x},t){\bf B}}{1+\gamma|{\bf B}|^2}, \nabla\times{\Phi})\nonumber\\
\label{equation:eq-37}
&+&({\bf U\times B},\nabla\times{\Phi}), \ \ \forall {\Phi\in \mathcal{V}},\\
(\partial_t\xi, \Upsilon)&+&(\kappa\nabla\xi, \nabla\Upsilon)
+<(\Psi(\xi+\theta_0)-\Psi(\theta_0)),\Upsilon>_{\Gamma_2}\nonumber\\
\label{equation:eq-38}
&=&
([q(\xi)\mathcal{K}({\bf B})]_\epsilon,\Upsilon)-(\kappa\nabla\theta_0,\nabla\Upsilon),\forall \Upsilon\in
 \mathcal{Y},
\end{eqnarray}
where $[\mathcal{D}]_\epsilon$ is the cut-off of $\mathcal{D}$ defined by
\begin{eqnarray*}
[\mathcal{D}]_\epsilon=\frac{\mathcal{D}}{1+\epsilon|\mathcal{D}|}, \ \ \epsilon>0.
\end{eqnarray*}
It is clear that
$[\mathcal{D}]_\epsilon\in L^\infty(\Omega)$. If $\mathcal{D}\in L^p(\Omega)$, the
$$\lim\limits_{\epsilon\longrightarrow 0} \|[\mathcal{D}]_\epsilon
- \mathcal{D} \|_{L^{p/2}(\Omega)}=0.$$

\subsection{Semi-discrete Approximation}
We will use Roth's method \cite{MR10} to explore the well-posedness of
solution of the regularized problem
(\ref{equation:eq-37})-(\ref{equation:eq-38}). Let $ N$ be a positive integer
 and let  an equidistant partition  of $[0,T]$ be given by
 $$t_n=n\tau, n=0,1,2,\cdots, N, \ \ \tau=T/N.$$
The semi-discrete approximation to (\ref{equation:eq-37})-(\ref{equation:eq-38})
can be formulated by : for
$\forall {\Phi\in \mathcal{V}}, \Upsilon\in
 \mathcal{Y},$ find $ {\bf B}^n\in \mathcal {V}$  and
$\xi^n\in \mathcal{Y}, 1\leq n \leq N$  with initial data ${\bf B}^0={\bf B}_0({\bf x}), \xi^0=0$  such that,
\begin{eqnarray}
(\frac{{\bf B}^n-{\bf B}^{n-1}}{\tau},{\Phi})&+&
(\lambda(\xi^{n-1}+\theta_0)\nabla\times {\bf B}^n,\nabla\times{\Phi})
+{\Lambda}(\nabla\cdot{\bf B}^n,\nabla\cdot{\Phi})
\nonumber\\
\label{equation:eq-39}
&=&R_\alpha(\frac{f({\bf x},n\tau){\bf B}^n}{1+\gamma|{\bf B}^{n-1}|^2}, \nabla\times{\Phi})
+({\bf U\times }{\bf B}^{n},\nabla\times{\Phi}), \\
(\frac{\xi^n-\xi^{n-1}}{\tau}, \Upsilon)&+&(\kappa\nabla\xi^n, \nabla\Upsilon)
+<(\Psi(\xi^n+\theta_0)-\Psi(\theta_0)),\Upsilon>_{\Gamma_2}\nonumber\\
\label{equation:eq-40}
&=&
(q(\xi^{n-1})[\mathcal{K}({\bf B}^{n})]_\epsilon,\Upsilon)-(\kappa\nabla\theta_0,\nabla\Upsilon).
\end{eqnarray}
For convenience, we also denote the difference operator
$$\delta_\tau w=\frac{w^n-w^{n-1}}{\tau},\  in\ \ [t_{n-1},t_n].$$
Obviously, (\ref{equation:eq-39})-(\ref{equation:eq-40}) can be solved
sequentially since (\ref{equation:eq-39}) is independent of (\ref{equation:eq-40})
 for a given ${\bf B}^{n-1}$ and (\ref{equation:eq-40}) can be solved after given by
 ${\bf B}^{n}$ in (\ref{equation:eq-39}) and $\xi^{n-1}$.

\subsection{Well-posedness of  the Nonlinear Magnetic Equation}

Let $\tilde{\bf B}_\tau$ and ${\bf B}_\tau$ denote the piecewise constant
and piecewise linear interpolations using the discrete solutions, that is
\begin{eqnarray}
\label{equation:eq-41}
\tilde{\bf B}_\tau(\cdot,t)={\bf B}^n, {\bf B}_\tau(\cdot, t)
=L_n(t){\bf B}^n+(1-L_n(t)){\bf B}^{n-1},
\end{eqnarray}
for any $t\in [t_{n-1},t_n]$ and $1\leq n\leq N$ with $L_n(t)=(t-t_{n-1})/\tau$.
Obviously, we have
$$ \tilde{\bf B}_\tau\in L^2(0,T;\mathcal{V}),
{\bf B}_\tau\in C(0,T;\mathcal{V}).$$
We also denote $$\hat{\bf B}_\tau=\tilde{\bf B}_\tau(:,t-\tau),
 \ \ \forall t\in(t_{n-1},t_n].$$

Let $\tilde{\xi}_\tau$ and ${\xi}_\tau$ denote the piecewise constant
and piecewise linear interpolations using the discrete solutions, that is
\begin{eqnarray}
\label{equation:eq-41}
\tilde{\xi}_\tau(\cdot,t)={\xi}^n, {\xi}_\tau(\cdot, t)
=L_n(t){\xi}^n+(1-L_n(t)){\xi}^{n-1},
\end{eqnarray}
for any $t\in [t_{n-1},t_n]$ and $1\leq n\leq N$.
We also denote $$\hat{\xi}_\tau=\tilde{\xi}_\tau(:,t-\tau)=\xi^{n-1},
 \ \ \forall t\in(t_{n-1},t_n].$$

\begin{theorem}
\label{theorem:thm-5}
For any $1\leq n\leq N$ and for given ${\bf B}^{n-1}$, the weak formula
(\ref{equation:eq-39}) has a unique solution ${\bf B}^{n}\in \mathcal{V}$.
For a given $ {\bf B}^{n}\in\mathcal{V} $ and $ \xi^{n-1}\in \mathcal{Y}$,
the weak formula (\ref{equation:eq-40}) has a unique solution
$\xi^n\in \mathcal{Y}$.
\end{theorem}

\begin{proof}
We rewrite the weak formula (\ref{equation:eq-39}) as:
find ${\bf B}^{n}\in \mathcal{V}$ such that
\begin{eqnarray}
&&(\frac{{\bf B}^n}{\tau},{\Phi})+
(\lambda(\xi^{n-1}+\theta_0)\nabla\times {\bf B}^n,\nabla\times{\Phi})
+{\Lambda}(\nabla\cdot{\bf B}^n,\nabla\cdot{\Phi})\nonumber\\
&-&
R_\alpha(\frac{f({\bf x},n\tau){\bf B}^n}{1+\gamma|{\bf B}^{n-1}|^2}, \nabla\times{\Phi})
-({\bf U\times }{\bf B}^{n},\nabla\times{\Phi})
\nonumber\\
\label{equation:eq-42}
&=&(\frac{{\bf B}^{n-1}}{\tau},{\Phi}),\ \ \forall \Phi\in \mathcal{V},
\end{eqnarray}
which is equivalent to an nonlinear operator equation
\begin{eqnarray}
\label{equation:eq-43}
\mathcal{P}{\bf B}^n=F_{n-1},
\end{eqnarray}
where $F_{n-1}\in \mathcal{V}'$ defined by
$ (F_{n-1},\Phi)=(\frac{{\bf B}^{n-1}}{\tau},{\Phi}). $

We also rewrite (\ref{equation:eq-40}) as:
find $\xi^{n}\in \mathcal{Y}$ such that
\begin{eqnarray}
&&(\frac{\xi^n}{\tau}, \Upsilon)+(\kappa\nabla\xi^n, \nabla\Upsilon)
+<(\Psi(\xi^n+\theta_0)-\Psi(\theta_0)),\Upsilon>_{\Gamma_2}\nonumber\\
\label{equation:eq-44}
&&=(\frac{\xi^{n-1}}{\tau}, \Upsilon)+
(q(\xi^{n-1})[\mathcal{K}({\bf B}^{n})]_\epsilon,\Upsilon)-(\kappa\nabla\theta_0,\Upsilon),
 \forall \Upsilon\in \mathcal{Y},
\end{eqnarray}
which is equivalent to an nonlinear operator equation
\begin{eqnarray}
\label{equation:eq-45}
\mathcal{L}{\xi}^n=H_{n-1},
\end{eqnarray}
where $H_{n-1}\in \mathcal{Y}'$ defined by
$$ (H_{n-1},\Upsilon)=(\frac{\xi^{n-1}}{\tau}, \Upsilon)+
(q(\xi^{n-1})[\mathcal{K}({\bf B}^{n})]_\epsilon,\Upsilon)-(\kappa\nabla\theta_0,\Upsilon),
 \forall \Upsilon\in \mathcal{Y}. $$

From Lemma \ref{lemma:Lem-1}-Lemma \ref{lemma:Lem-4}, we know that
$\mathcal{P}$ and $\mathcal{L}$  are a bounded, coercive, strictly
 monotone,and semi-continuous
operator on $\mathcal{V}  $ and $\mathcal{Y}$, respectively.
From \cite{MR11,MR12}, we know that problem (\ref{equation:eq-42})
has a solution ${\bf B}^{n}\in \mathcal{V}$, and (\ref{equation:eq-44}).
has a solution $\xi^n\in \mathcal{Y}$.

\par
Now we have to prove the uniqueness of the solution.
 Let ${\bf B}^{n},\check{\bf B}^{n}$ be the two solutions
 of (\ref{equation:eq-42}).
 From Lemma \ref{lemma:Lem-3}, we have
 $$0=<\mathcal{P}{\bf B}^{n}-\mathcal{P}\check{\bf B}^{n},
 {\bf B}^{n}-\check{\bf B}^{n}>
 \geq C_6\|{\bf B}^{n}-\check{\bf B}^{n}\|_{\mathcal{V}}^2.$$
 We can conclude  $ {\bf B}^{n}=\check{\bf B}^{n}$ in $\Omega$,
 which means the uniqueness of the solution of (\ref{equation:eq-42}).
Let $\xi^n,\check{\xi}^n$ be the two solutions of (\ref{equation:eq-44}).
 From Lemma \ref{lemma:Lem-3}, we also have
$$ 0=<\mathcal{L}\xi^n-\mathcal{L}\check{\xi}^n,\xi^n-\check{\xi}^n>
\geq C_7\|\xi^n-\check{\xi}^n\|_1^2
+\frac{\zeta}{8}\|\xi^n-\check{\xi}^n\|_{L^5(\Gamma_2)}^2,$$
 which means the uniqueness of the solution of (\ref{equation:eq-44}).
\end{proof}

\begin{lemma}
\label{lemma:Lem-6}
There exists two positive constants $C_8$ and $C$  dependent
of $ R_\alpha$, $ \|f({\bf x},t)\|_{L^\infty(0,T;L^\infty(\Omega))}$,
$\|{\bf U}\|_{L^\infty(0,T;L^\infty(\Omega))}$ such that
\begin{eqnarray}
\label{equation:eq-47}
&&\|{\bf B}^n\|_0^2
+\sum_{i=1}^n\tau\lambda_0\| \nabla\times {\bf B}^i\|_0^2
+\sum_{i=1}^n\tau\Lambda\| \nabla\cdot {\bf B}^i\|_0^2
\leq C_8\|{\bf B}^0\|_0^2.\\
\label{equation:eq-48}
&& \|{\bf B}_\tau\|_{L^\infty(0,T;L^2(\Omega))}
+\sqrt{\lambda_0}\| \nabla\times {\bf B}_\tau\|_{L^2(0,T;L^2(\Omega))}
+\Lambda\| \nabla\cdot {\bf B}_\tau\|_{L^2(0,T;L^2(\Omega))}
\leq C,\\
\label{equation:eq-49}
&&\|\tilde{\bf B}_\tau\|_{L^\infty(0,T;L^2(\Omega))}
+\sqrt{\lambda_0}\| \nabla\times \tilde{\bf B}_\tau\|_{L^2(0,T;L^2(\Omega))}
+\Lambda\| \nabla\cdot \tilde{\bf B}_\tau\|_{L^2(0,T;L^2(\Omega))}
\leq C.
\end{eqnarray}

\end{lemma}

\begin{proof}
Taking $\Phi={\bf B}^n$ in (\ref{equation:eq-39}), we have
\begin{eqnarray}
&&({\bf B}^n-{\bf B}^{n-1},{\bf B}^n)
+\tau\lambda_0(\nabla\times {\bf B}^n,\nabla\times {\bf B}^n)
+\tau\Lambda(\nabla\cdot {\bf B}^n,\nabla\cdot {\bf B}^n)\nonumber\\
\label{equation:eq-50}
&\leq&\tau R_\alpha(\frac{f({\bf x},t)}{1+\gamma|{\bf B}^{n-1}|^2}{\bf B}^n, \nabla\times {\bf B}^n   )
+\tau({\bf U}\times {\bf B}^n, \nabla\times {\bf B}^n).
\end{eqnarray}
Since
$$2({\bf B}^n-{\bf B}^{n-1},{\bf B}^n)
\geq \|{\bf B}^n\|_0^2-\|{\bf B}^{n-1}\|_0^2,$$
summing up (\ref{equation:eq-50}) from $i=1,2,\cdots,n$, we have
\begin{eqnarray}
&&\frac{1}{2}(\|{\bf B}^n\|_0^2-\|{\bf B}^0\|_0^2)
+\sum_{i=1}^n\tau\lambda_0\| \nabla\times {\bf B}^i\|_0^2
+\sum_{i=1}^n\tau\Lambda\| \nabla\cdot {\bf B}^i\|_0^2\nonumber\\
&\leq&\sum_{i=1}^n(\tau R_\alpha\|f({\bf x},t)\|_{L^\infty(\Omega)}\|{\bf B}^i\|_0
\| \nabla\times {\bf B}^i\|_0)\nonumber\\
&+&\sum_{i=1}^n(\tau \|{\bf U}\|_{L^\infty(\Omega)}\|{\bf B}^i\|_0
\| \nabla\times {\bf B}^i\|_0)\nonumber\\
&\leq&\sum_{i=1}^n\tau R_\alpha\|f({\bf x},t)\|_{L^\infty(\Omega)}
(\frac{\|{\bf B}^i\|_0^2}{4\delta_i}+\delta_i
\| \nabla\times {\bf B}^i\|_0^2)\nonumber\\
\label{equation:eq-51}
&+&\sum_{i=1}^n\tau\|{\bf U}\|_{L^\infty(\Omega)}
(\frac{\|{\bf B}^i\|_0^2}{4\bar{\delta}_i}+\bar{\delta}_i
\| \nabla\times {\bf B}^i\|_0^2).
\end{eqnarray}
Taking $ \delta_i, \bar{\delta}_i$ such that
$$ \frac{1}{2\tau}-\frac{R_\alpha\|f({\bf x},t)\|_{L^\infty(\Omega)}}{4\delta_n}
-\frac{\|{\bf U}\|_{L^\infty(\Omega)}}{4\bar{\delta}_n}>0,$$
 $$\lambda_0-R_\alpha\|f({\bf x},t)\|_{L^\infty(\Omega)}\delta_i
-\|{\bf U}\|_{L^\infty(\Omega)}\bar{\delta}_i>0,$$
based on Growall's inequality, we have
\begin{eqnarray}
\label{equation:eq-52}
\|{\bf B}^n\|_0^2
+\sum_{i=1}^n\tau\lambda_0\| \nabla\times {\bf B}^i\|_0^2
+\sum_{i=1}^n\tau\Lambda\| \nabla\cdot {\bf B}^i\|_0^2
\leq C_8\|{\bf B}^0\|_0^2,
\end{eqnarray}
where $C_8$ is independent of $ n$, which means
$$\|{\bf B}_\tau\|_{L^\infty(0,T;L^2(\Omega))}
+\sqrt{\lambda_0}\| \nabla\times {\bf B}_\tau\|_{L^2(0,T;L^2(\Omega))}
+\Lambda\| \nabla\cdot {\bf B}_\tau\|_{L^2(0,T;L^2(\Omega))}
\leq C.$$
Similarly, (\ref{equation:eq-49})  comes directly from (\ref{equation:eq-48}).
\end{proof}


\begin{lemma}
\label{lemma:Lem-7}
There exists two  positive constants and $C_9$ and  $C$  dependent
of $ \epsilon, \kappa, q_{\max},\zeta,\|\xi^0\|_{0},\|\nabla\theta_0\|_0$ such that
\begin{eqnarray}
\label{equation:eq-53}
&&\|{\xi}^n\|_0^2+\sum_{i=1}^n\tau\|\kappa \nabla \xi^i\|_0^2
+\sum_{i=1}^n\tau\omega\|\xi^i\|_{L^2(\Gamma_2)}^2
\leq C_9\|{\xi}^{0}\|_0^2+\|\nabla\theta_0\|_0.\\
\label{equation:eq-54}
&&\|\xi_\tau\|_{L^\infty(0,T;L^2(\Omega))}
+\|\nabla\xi_\tau\|_{L^2(0,T;L^2(\Omega))}
+\frac{\zeta}{8}\|\xi_\tau\|_{L^2(0,T;L^5(\Gamma_2))}\leq C,\\
\label{equation:eq-55}
&&\|\tilde{\xi}_\tau\|_{L^\infty(0,T;L^2(\Omega))}
+\|\nabla\tilde{\xi}_\tau\|_{L^2(0,T;L^2(\Omega))}
+\frac{\zeta}{8}\|\tilde{\xi}_\tau\|_{L^2(0,T;L^5(\Gamma_2))}\leq C.
\end{eqnarray}

\end{lemma}

\begin{proof}
\par
Taking $\Upsilon=\xi^n$ in (\ref{equation:eq-40}), we have

\begin{eqnarray}
({\xi^n-\xi^{n-1}}, \xi^n)&+&{\tau}(\kappa\nabla\xi^n, \nabla\xi^n)
+{\tau}<(\Psi(\xi^n+\theta_0)-\Psi(\theta_0)),\xi^n>_{\Gamma_2}\nonumber\\
\label{equation:eq-56}
&=&
{\tau}(q(\xi^{n-1})[\mathcal{K}({\bf B}^{n})]_\epsilon,\xi^n)-\tau(\kappa\nabla\theta_0,\nabla\Upsilon).
\end{eqnarray}
Firstly, we estimate the right hand of (\ref{equation:eq-56}). We should
notice
\begin{eqnarray}
\label{equation:eq-57}
|[\mathcal{K}({\bf B}^{n})]_\epsilon|=\frac{|\mathcal{K}({\bf B}^{n})|}{1+\epsilon|\mathcal{K}({\bf B}^{n})|}
\leq \frac{1}{\epsilon},
\end{eqnarray}
which leads to
\begin{eqnarray}
\label{equation:eq-58}
([q(\xi^{n-1})\mathcal{K}({\bf B}^{n})]_\epsilon,\xi^n)
\leq \frac{1}{\epsilon}\|\xi^n\|_0.
\end{eqnarray}
And we also have
$$ \tau|(\kappa\nabla\theta_0,\nabla\xi^n)|\leq \kappa\tau\|\nabla\theta_0\|_0\|\nabla\xi^n\|_0. $$

Since
$$2({\xi}^n-{\xi}^{n-1},{\xi}^n)
\geq \|{\xi}^n\|_0^2-\|{\xi}^{n-1}\|_0^2,$$
and
$$(\Psi(\xi^n+\theta_0)-\Psi(\theta_0),\xi^n)
\geq \frac{\zeta}{8}\|\xi^n\|^5_{L^5(\Gamma_2)}
+\omega \|\xi^n\|^2_{L^2(\Gamma_2)}\geq \frac{\zeta}{8} \|\xi^n\|^5_{L^5(\Gamma_2)},$$
summing up (\ref{equation:eq-56}) from $i=1,2,\cdots,n$, we have
\begin{eqnarray}
&&\frac{1}{2}(\|{\xi}^n\|_0^2-\|{\xi}^{0}\|_0^2)
+\sum_{i=1}^n\tau\|\kappa \nabla \xi^i\|_0^2
+\sum_{i=1}^n \frac{\zeta\tau}{8}\|\xi^i\|_{L^5(\Gamma_2)}^5\nonumber\\
\label{equation:eq-59}
&\leq & \sum_{i=1}^n(\frac{\tau }{\epsilon}\|\xi^i\|_0+\kappa\tau\|\nabla\theta_0\|_0\|\nabla\xi^i\|_0)\nonumber\\
&\leq &\sum_{i=1}^n (\frac{\tau^2}{4\epsilon^2\hat{\delta}_i}
+\hat{\delta}_i\|\xi^i\|_0^2+\frac{\kappa\tau}{4\varsigma}\|\nabla\theta_0\|_0^2
+\kappa\tau\varsigma\|\nabla\xi^i\|_0^2)
\end{eqnarray}
Taking $\hat{\delta}_n$ and $ \varsigma$ such that $\hat{\delta}_n+\kappa\tau\varsigma\leq \frac{1}{2}$,
and using Growall's inequality, we have
\begin{eqnarray}
\label{equation:eq-60}
\|{\xi}^n\|_0^2+\sum_{i=1}^n\tau\|\kappa \nabla \xi^i\|_0^2
+\sum_{i=1}^n\frac{\zeta\tau}{8}\|\xi^i\|_{L^5(\Gamma_2)}^5
\leq C_9(\|{\xi}^{0}\|_0^2+\|\nabla\theta_0\|_0^2),
\end{eqnarray}
where $C_9$ is independent of $n$, which means
$$ \|\xi_\tau\|_{L^\infty(0,T;L^2(\Omega))}
+\|\nabla\xi_\tau\|_{L^2(0,T;L^2(\Omega))}
+\frac{\zeta}{8}\|\xi_\tau\|_{L^2(0,T;L^5(\Gamma_2))}\leq C.$$
Similarly, (\ref{equation:eq-55})  comes directly from (\ref{equation:eq-54}).
\end{proof}

\subsection{The Existence of the Solution of the Regularized Problem}

From Lemma \ref{lemma:Lem-6}-Lemma \ref{lemma:Lem-7}, we can see that
 the two discrete solutions of  the Regularized Problem imply the boundedness in
 $\mathcal{V}$ and $\mathcal{Y}$, respectively.  Since
  both $\mathcal{V}$ and $\mathcal{Y}$ are  reflexive, there exist a subsequence of
  $\tilde{\bf B}_\tau$ and a subsequence of ${\bf B}_\tau$, which have
  common subscripts and are denoted the same notations such that
  $$\tilde{\bf B}_\tau\rightharpoonup {\bf B},\ \
  {\bf B}_\tau\rightharpoonup {\bf B},\ in \ \ L^2(0,T;\mathcal{V}),$$
  where $ {\bf B}\in L^2(0,T;\mathcal{V})$ and "$\rightharpoonup $" denote the weak
  convergence of the sequences.

  Similarly,there exist a subsequence of
  $\tilde{\xi}_\tau$ and a subsequence of ${\xi}_\tau$ with the same
subscripts such that
$$\tilde{\xi}_\tau \rightharpoonup \xi,\ \
{\xi}_\tau\rightharpoonup\xi , \  in \ \ L^2(0,T;\mathcal{Y}).$$
Furthermore, based on $L^2(Q_T)$ is embedded compactly into $L^1(Q_T)$, we
have
\begin{eqnarray}
\label{equation:eq-61}
&&\tilde{\bf B}_\tau\longrightarrow {\bf B},
\nabla\times \tilde{\bf B}_\tau\longrightarrow \nabla\times{\bf B},
\ \ \tilde{\xi}_\tau \longrightarrow \xi,\ \  in\ \ L^1(Q_T),\\
\label{equation:eq-62}
&&{\bf B}_\tau\longrightarrow {\bf B},
\nabla\times {\bf B}_\tau\longrightarrow \nabla\times{\bf B},
\ \ \ {\xi}_\tau \longrightarrow \xi,\ \  in\ \ L^1(Q_T),
\end{eqnarray}
where "$ \longrightarrow $" means strong convergence of the sequences.
In this subsection, we shall present the proof that the limits $ {\bf B}, \xi$
solve the regularized problem (\ref{equation:eq-37})- (\ref{equation:eq-38}).
Without  causing the confusions, we always use $\{ \tilde{\bf B}_\tau\},
\{ {\bf B}_\tau\},\{ \tilde{\xi}_\tau\},\{ \xi_\tau\}$ to denote their convergence subsequences
in the  rest of this section.

\begin{theorem}
The limit function ${\bf B}$ is the weak solution of problem
(\ref{equation:eq-37}) with initial dada
 ${\bf B}({\bf x},0)={\bf B}_0({\bf x})$.

\end{theorem}
\begin{proof}
\par
Remember that $C_0^\infty(\Omega)\subset \mathcal{V}$. For any
${\bf v}({\bf x},t)={ \Phi}({\bf x}) \phi(t)$ with ${ \Phi}({\bf x})\in C_0^\infty(\Omega)\ $
 and $\phi(t)\in C_0^\infty(0,T)$, from (\ref{equation:eq-39})
 after calculating, we have
  \begin{eqnarray}
 \label{equation:eq-63}
 \int_{Q_T}\frac{ \partial {\bf B}_\tau}{\partial t} {\bf v}
 &+&\int_{Q_T}\lambda(\hat{\xi}_\tau+\theta_0) \nabla\times \tilde{\bf B}_\tau \nabla\times {\bf v}
 +\int_{Q_T}\Lambda \nabla\cdot \tilde{\bf B}_\tau \cdot\nabla\cdot{\bf v}\nonumber\\
 &=&R_\alpha \int_{Q_T}\frac{f({\bf x},t)}{1+\gamma|\hat{\bf B}_\tau|^2}
 \tilde{\bf B}_\tau \cdot\nabla\times {\bf v}
 +\int_{Q_T}({\bf U}\times \tilde{\bf B}_\tau)\cdot \nabla\times {\bf v}.
 \end{eqnarray}
 For the first term, we have
 \begin{eqnarray}
 \label{equation:eq-64}
 \lim\limits_{\tau\rightarrow 0}\int_{Q_T}\frac{ \partial {\bf B}_\tau}{\partial t} {\bf v}
 =-\lim\limits_{\tau\rightarrow 0}\int_{Q_T}{\bf B}_\tau
 \frac{ \partial {\bf v}}{\partial t}
 =-\int_{Q_T}{\bf B}\cdot
 \frac{ \partial {\bf v}}{\partial t}=\int_{Q_T}
 \frac{ \partial {\bf B}}{\partial t}\cdot {\bf v}.
 \end{eqnarray}
  For the second and the third term, since
  $\nabla\times \tilde{\bf B}_\tau $ and $\nabla\cdot \tilde{\bf B}_\tau$
  converge to $\nabla\times {\bf B}$ and $\nabla\cdot {\bf B}$, respectively,
   we have
    \begin{eqnarray*}
 \lim\limits_{\tau\longrightarrow 0}
 \int_{Q_T}\lambda(\hat{\xi}_\tau+\theta_0) \nabla\times \tilde{\bf B}_\tau \cdot\nabla\times {\bf v}
& = &\int_{Q_T}\lambda({\xi}+\theta_0) \nabla\times{\bf B}\cdot \nabla\times {\bf v},\\
  \lim\limits_{\tau\longrightarrow 0}
 \int_{Q_T}\Lambda \nabla\cdot\tilde{\bf B}_\tau \cdot\nabla\cdot {\bf v}
 &=& \int_{Q_T}\Lambda \nabla\cdot{\bf B}\cdot \nabla\cdot {\bf v}.\ \
 \end{eqnarray*}
Since both $f({\bf x},t)$and $ {
\bf U}$ are Lipschitz continuous and
$\frac{1}{1+\gamma|\hat{\bf B}_\tau|^2}\longrightarrow
\frac{1}{1+\gamma|{\bf B}|^2}$ and
${\bf U}\times \tilde{\bf B}_\tau \longrightarrow {\bf U}\times {\bf B} $
strongly in $L^1(Q_T)$, we have
\begin{eqnarray}
\label{equation:eq-65}
&&\lim\limits_{\tau\longrightarrow 0}\int_{Q_T}\frac{1}{1+\gamma|\hat{\bf B}_\tau|^2}
\tilde{\bf B}_\tau\cdot \nabla\times {\bf v}=
\int_{Q_T}\frac{1}{1+\gamma|{\bf B}|^2}
{\bf B}\cdot \nabla\times {\bf v},\\
\label{equation:eq-66}
&&\lim\limits_{\tau\longrightarrow 0}\int_{Q_T}
{\bf U}\times \tilde{\bf B}_\tau \cdot \nabla\times{\bf v}
\longrightarrow \int_{Q_T}{\bf U}\times {\bf B} \cdot \nabla\times{\bf v}.
\end{eqnarray}
From (\ref{equation:eq-64})-(\ref{equation:eq-66}), we have
  \begin{eqnarray}
 \label{equation:eq-67}
 \int_{Q_T}\frac{ \partial {\bf B}}{\partial t} {\bf v}
 &+&\int_{Q_T}\lambda ({\xi}+\theta_0)\nabla\times {\bf B} \nabla\times {\bf v}
 +\int_{Q_T}\Lambda \nabla\cdot {\bf B} \cdot\nabla\times {\bf v}\nonumber\\
 &=&R_\alpha \int_{Q_T}\frac{f({\bf x},t)}{1+\gamma|{\bf B}|^2}
{\bf B} \cdot\nabla\times {\bf v}
 +\int_{Q_T}({\bf U}\times {\bf B})\cdot \nabla\times {\bf v}.
 \end{eqnarray}
Since $\phi(t)\in C_0^\infty(0,T)$,  for any ${ \Phi}\in C_0^\infty(\Omega)$,
we have
  \begin{eqnarray}
 \label{equation:eq-68}
 \int_{\Omega}\frac{ \partial {\bf B}}{\partial t} { \Phi}
 &+&\int_{\Omega}\lambda({\xi}+\theta_0) \nabla\times {\bf B} \nabla\times { \Phi}
 +\int_{\Omega}\Lambda \nabla\cdot {\bf B} \cdot\nabla\times { \Phi}\nonumber\\
 &=&\int_{\Omega}R_\alpha \int_{Q_T}\frac{f({\bf x},t)}{1+\gamma|{\bf B}|^2}
{\bf B} \cdot\nabla\times { \Phi}
 +\int_{\Omega}({\bf U}\times {\bf B})\cdot \nabla\times { \Phi}.
 \end{eqnarray}
By the density of $C_0^\infty $ in $\mathcal{V}$, the equation
(\ref{equation:eq-68}) holds
for any ${ \Phi}\in \mathcal{V}$, too.
\end{proof}

\begin{theorem}
The limit function $\xi$ is the weak solution of problem (\ref{equation:eq-38})
 with the initial condition $\xi({\bf x},0)=0$.

\end{theorem}
\begin{proof}
For any $\eta({\bf x},t)=\Upsilon({\bf x})\phi(t)$ with
$\Upsilon({\bf x})\in C^\infty_0(\Omega)$ and $\phi(t)\in C_0^\infty(0,T), $
from (\ref{equation:eq-40}) we know that
\begin{eqnarray}
\int_{Q_T}\frac{\partial \xi_\tau}{\partial t} \eta
&+&\int_{Q_T}\kappa\nabla \tilde{\xi}_\tau\cdot \nabla\eta
+\int_0^T\int_{\Gamma_2}(\Psi(\tilde{\xi}_\tau+\theta_0)-\Psi(\theta_0))\eta\nonumber\\
\label{equation:eq-69}
&=&
\int_{Q_T}q(\hat{\xi}_\tau)[\mathcal{K}(\tilde{\bf B}_\tau)]_\epsilon\eta-\int_{Q_T}\kappa\nabla\theta_0\cdot \nabla\eta.
\end{eqnarray}

Similar to (\ref{equation:eq-64}), and the weak convergence
in $L^2(0,T;\mathcal{Y})$,
we have
\begin{eqnarray}
\label{equation:eq-70}
\lim\limits_{\tau\longrightarrow 0}\int_{Q_T}\frac{\partial \xi_\tau}{\partial t}
\eta=\int_{Q_T}\frac{\partial \xi}{\partial t}\eta, \ \
\lim\limits_{\tau\longrightarrow 0}\int_{Q_T}
\kappa\nabla \tilde{\xi}_\tau\cdot \nabla\eta \longrightarrow
\int_{Q_T}
\kappa\nabla {\xi}\cdot \nabla\eta.
\end{eqnarray}
Since $ L^5(\Gamma_2)$ is embedded compactly into $ L^4(\Gamma_2)$ and
$\tilde{\xi}_\tau$ converges strongly to $\xi$ in  $ L^4(\Gamma_2)$, we have
\begin{eqnarray}
\label{equation:eq-71}
\lim\limits_{\tau\longrightarrow 0}\int_0^T\int_{\Gamma_2}
(\Psi(\tilde{\xi}_\tau+\theta_0)-\Psi(\theta_0))\eta
\longrightarrow
\int_0^T\int_{\Gamma_2}
(\Psi({\xi}+\theta_0)-\Psi(\theta_0))\eta.
\end{eqnarray}
The right hand of the equation (\ref{equation:eq-69}) can be considered
 from the the uniform boundedness and the strong convergence  of $\hat{\xi}_\tau$ in
 $L^2(0,T,\mathcal{Y})$ and the strong convergence  of
  $ \nabla\times \tilde{\bf B}_\tau \rightarrow \nabla\times {\bf B}$,
  $ {\bf U}\times \tilde{\bf B}_\tau \rightarrow {\bf U}\times {\bf B}$,
  $ \frac{f({\bf x},t)\tilde{\bf B}_\tau}{1+\gamma|\hat{\bf B}_\tau|^2}
  \rightarrow \frac{f({\bf x},t){\bf B}}{1+\gamma|{\bf B}|^2}$
in $L^1(Q_T)$, we have
\begin{eqnarray}
\label{equation:eq-72}
\lim\limits_{\tau\rightarrow 0}
\int_{Q_T}[q(\hat{\xi}_\tau)\mathcal{K}(\tilde{\bf B}_\tau)]_\epsilon\eta
\rightarrow
\int_{Q_T}[q({\xi})\mathcal{K}({\bf B})]_\epsilon\eta.
\end{eqnarray}
From (\ref{equation:eq-69})-(\ref{equation:eq-72}), we can get
\begin{eqnarray}
\label{equation:eq-73}
\int_{Q_T}\frac{\partial \xi}{\partial t}\eta&+&
\int_{Q_T}\kappa\nabla {\xi}\cdot \nabla\eta+\int_0^T\int_{\Gamma_2}
(\Psi({\xi}+\theta_0)-\Psi(\theta_0))\eta\nonumber\\
&=&
\int_{Q_T}[q({\xi})\mathcal{K}({\bf B})]_\epsilon\eta-\int_{Q_T}\kappa\nabla\theta_0\cdot \nabla\eta.
\end{eqnarray}
By the arbitrariness of $\phi(t)$, it yields
\begin{eqnarray}
\label{equation:eq-74}
\int_{\Omega}\frac{\partial \xi}{\partial t}\Upsilon&+&
\int_{\Omega}\kappa\nabla \tilde{\xi}\cdot \nabla\Upsilon+\int_{\Gamma_2}
(\Psi({\xi}+\theta_0)-\Psi(\theta_0))\Upsilon\nonumber\\
&=&
\int_{\Omega}q({\xi})\mathcal{K}({\bf B})\Upsilon-\int_{\Omega}\kappa\nabla\theta_0\cdot \nabla\Upsilon,\ \ \forall \Upsilon \in
\mathcal{Y}\cap C^\infty(\Omega).
\end{eqnarray}
By the density of $\mathcal{Y}\cap C^\infty(\Omega)$ in $\mathcal{Y}$,
the above equation (\ref{equation:eq-74}) holds for any
 $\Upsilon \in \mathcal{Y}.$

\par
Taking any $\Upsilon\in C_0^\infty$ and let $\eta(t)=(T-t)\Upsilon$, we
have $\eta(0)=T\Upsilon, \eta(T)=0.$ Using integration by part, we have
\begin{eqnarray*}
&&T\int_\Omega \xi(0)\cdot \Upsilon({\bf x})
=-\int_0^T\int_\Omega\frac{\partial}{\partial t}(\xi \cdot \eta)
=\int_0^T\int_\Omega \xi\cdot \Upsilon-\int_0^T\int_\Omega
\frac{\partial \xi}{\partial t}\cdot \eta\\
&=&\lim\limits_{\tau\rightarrow 0}\int_0^T\int_\Omega {\xi}_\tau\cdot \Upsilon
+\int_0^T\int_{\Gamma_2}(\Psi(\xi+\theta_0)-\Psi(\theta_0))\eta
+\int_0^T\int_\Omega
[\kappa \nabla\xi \cdot \nabla \eta- q(\xi)\mathcal{K}({
\bf B})\cdot\eta]\\
&&=\lim\limits_{\tau\rightarrow 0}\int_0^T\int_\Omega [{\xi}_\tau\cdot \Upsilon
 +[\kappa \nabla\tilde{\xi}_\tau \cdot \nabla \eta- q(\hat{\xi}_\tau)
 \mathcal{K}(\tilde{
\bf B}_\tau)\cdot\eta]
+\lim\limits_{\tau\rightarrow 0}\int_0^T\int_{\Gamma_2}
(\Psi(\tilde{\xi}_\tau+\theta_0)-\Psi(\theta_0))\cdot\eta\\
&&=\lim\limits_{\tau\rightarrow 0}[\int_0^T\int_\Omega
\xi_\tau\cdot \Upsilon-\int_0^T\int_\Omega
\frac{\partial \xi_\tau}{\partial t}\cdot \eta]=\lim\limits_{\tau\rightarrow 0}
T\int_\Omega\xi_\tau(0)\cdot\Upsilon({\bf x})=0, \forall \Upsilon({\bf x})
\in C_0^\infty(\Omega).
\end{eqnarray*}
Therefore, $\xi({\bf x},0)=0,$  which finish the proof.

\end{proof}

\subsection{Stability of the Regularized Problem}
Now we present the stability estimate of the regularized problem
 to ensure the well-posedness of the equations
 (\ref{equation:eq-19})- (\ref{equation:eq-20}).
 \begin{lemma}
 \label{lemma:Lem-71}
 There exists a constant $C_{10}$ depending of $\Omega, T,
 \lambda_0, \Lambda, R_\alpha,
\|f({\bf x},t)\|_{L^\infty(0,T,L^\infty(\Omega))}$ and
  $ \|{\bf U}\|_{L^\infty(0,T,L^\infty(\Omega))} $ such that
  \begin{eqnarray*}
\|{\bf B}\|_{L^\infty(0,T;L^2(\Omega))}+\sqrt{\lambda_0}\|\nabla\times{\bf B}\|_{L^2(0,T;L^2(\Omega))}
+\Lambda\|\nabla\cdot{\bf B}\|_{L^2(0,T;L^2(\Omega))}\leq C_{10}\|{\bf B}_0\|_0.
\end{eqnarray*}
 \end{lemma}
\begin{proof}
Taking $ \Phi={\bf B}$ in (\ref{equation:eq-37}) and using $\lambda(\xi+\theta_0)\geq \lambda_0>0$, we have
\begin{eqnarray}
\frac{1}{2}\frac{\partial}{\partial t}\|{\bf B}\|_0^2&+&
\lambda_0\|\nabla\times{\bf B}\|_0^2
+{\Lambda}\|\nabla\cdot{\bf B}\|_0^2\nonumber\\
&\leq&
R_\alpha(\frac{f({\bf x},t){\bf B}}{1+\gamma|{\bf B}|^2}, \nabla\times{\bf B})
\label{equation:eq-75}
+({\bf U\times B},\nabla\times{\bf B}).
\end{eqnarray}
Integrating the above with respect to $0\leq t\leq s$ and using Cauchy inequality
and young inequality, we have
\begin{eqnarray}
\|{\bf B}\|^2_0&+&\int_0^s(\lambda_0\|\nabla\times{\bf B}\|_0^2+{\Lambda}\|\nabla\cdot{\bf B}\|_0^2)\nonumber\\
&\leq& \frac{R_\alpha\|f({\bf x},t\|_{L^\infty(0,T,L^\infty(\Omega))}}{4a_1}
\int_0^s\|{\bf B}\|_0^2+\int_0^sa_1\|\nabla\times{\bf B}\|_0^2\nonumber\\
\label{equation:eq-76}
&+&\frac{\|{\bf U}\|_{L^\infty(0,T,L^\infty(\Omega))}}{4a_2}\int_0^s\|{\bf B}\|_0^2
+\int_0^sa_2\|\nabla\times{\bf B}\|_0^2+\|{\bf B}_0\|^2_0
\end{eqnarray}
Taking $a_1,a_2$ such that $\lambda_0-a_1-a_2>0$, by employing
the Growall's inequality, we have
\begin{eqnarray}
\label{equation:eq-77}
\|{\bf B}\|_{L^\infty(0,T;L^2(\Omega))}+\sqrt{\lambda_0}\|\nabla\times{\bf B}\|_{L^2(0,T;L^2(\Omega))}
+\sqrt{\Lambda}\|\nabla\cdot{\bf B}\|_{L^2(0,T;L^2(\Omega))}\leq C_{10}\|{\bf B}_0\|_0.
\end{eqnarray}
The proof is completed.

\end{proof}

\begin{lemma}
\label{lemma:lem-8}
There exits a constant $C_{11}$ depending on $\Omega, T,\kappa,
 \lambda, \Lambda, R_\alpha,
\|f({\bf x},t)\|_{L^\infty(0,T,L^\infty(\Omega))}$ and
  $ \|{\bf U}\|_{L^\infty(0,T,L^\infty(\Omega))} $ such that
  \begin{eqnarray}
\label{equation:eq-78}
  \|\xi\|_{L^\infty(0,T;L^1(\Omega))}+\|\xi\|_{L^\infty(0,T;L^4(\Gamma_2))}^4
  \leq C_{11}\|{\bf B}_0\|_0.
\end{eqnarray}

\end{lemma}

\begin{proof}
First we need define a function
$$h_\rho(s)=\frac{1}{\rho}sign(s)\min (|s|, \rho).$$
Obviously, $ h_\rho(s)$ is a bounded, absolutely continuous and increasing
 function
in $ \mathbf{R}$. for any $\eta \in \mathcal{Y},$ it can be calculate to get
$ h_\rho(\eta)\in \mathcal{Y}$ and
\begin{eqnarray}
\label{equation:eq-79} \nabla h_\rho(\eta)=\chi_{\{ {\bf x}\in \Omega: |\eta|<\rho  \} }
\frac{1}{\rho}\nabla \eta,
\ \ \lim\limits_{\rho\rightarrow 0}h_\rho(\eta)\rightarrow sign(\eta),
\ a.e.\ in\ \Omega,
\end{eqnarray}
where $\chi_{\{ {\bf x}\in \Omega: |\eta|<\rho  \} }  $
is the characteristic function.

\par
Taking $\Upsilon=h_\rho(\Upsilon)$ in (\ref{equation:eq-38}), we have
\begin{eqnarray}
(\partial_t\xi, h_\rho(\Upsilon))&+&(\kappa\nabla\xi, \nabla h_\rho(\Upsilon))
+<(\Psi(\xi+\theta_0)-\Psi(\theta_0)),h_\rho(\Upsilon)>_{\Gamma_2}\nonumber\\
\label{equation:eq-80}
&=&
([q(\xi)\mathcal{K}({\bf B})]_\epsilon,h_\rho(\Upsilon))-(\kappa\nabla\theta_0,\nabla h_\rho(\Upsilon)).
\end{eqnarray}
We need analyze (\ref{equation:eq-80}) term by term. From the convergence
 of (\ref{equation:eq-79}), we have
 \begin{eqnarray}
 \label{equation:eq-81}
 &&\lim\limits_{\rho\rightarrow 0}\int_\Omega\frac{\partial\xi}{\partial t}
  h_\rho(\xi)
 =\int_\Omega\frac{\partial\xi}{\partial t} sign(\xi)
 =\frac{\partial}{\partial t}\int_\Omega|\xi|,\\
  \label{equation:eq-82}
 &&\int_\Omega\kappa \nabla\xi\cdot \nabla h_\rho(\xi) =\frac{1}{\rho}
 \int_\Omega \chi_{\{ {\bf x}\in \Omega: |\xi|<\rho  \} }
 \kappa|\nabla\xi|^2 \geq 0,\\
  \label{equation:eq-83}
  &&\lim\limits_{\rho\rightarrow 0}\int_{\Gamma_2}
  (\Psi(\xi+\theta_0)-\Psi(\theta_0))h_\rho(\xi)
  \geq \frac{\zeta}{8}\|\xi\|_{L^4(\Gamma_2)}^4
  +\omega\|\xi\|_{L^1(\Gamma_2)},
 \end{eqnarray}
and
 \begin{eqnarray}
|\int_\Omega[&q&(\xi)\mathcal{K}({\bf B})]_\epsilon h_\rho(\xi)|
\leq |\int_\Omega q(\xi)\mathcal{K}({\bf B})|
\leq C(\|\nabla\times{\bf B}\|_0^2\nonumber\\
&+&\frac{R_\alpha\|f({\bf x},t)\|_{L^\infty(0,T;L^\infty(\Omega))}}{4a_3}\|{\bf B}\|_0^2
+a_3\|\|\nabla\times{\bf B}\|_0^2\nonumber\\
\label{equation:eq-84}
&+&\frac{\|{\bf U}\|_{L^\infty(0,T;L^\infty(\Omega))}}{4a_4}\|{\bf B}\|_0^2
+a_4\|\|\nabla\times{\bf B}\|_0^2.
 \end{eqnarray}
From (\ref{equation:eq-81})- (\ref{equation:eq-84}) and Lemma \ref{lemma:Lem-71}, we have
 \begin{eqnarray}
 \label{equation:eq-85}
\frac{\partial}{\partial t}\int_\Omega|\xi|+
\frac{\zeta}{8}\|\xi\|_{L^4(\Gamma_2)}^4
\leq C(\|{\bf B}_0\|_0+\|\nabla\theta_0\|_0).
 \end{eqnarray}
By using The Growall's inequality and integrating the inequality with
respect to $0\leq t\leq s$ for any $s\in [0,T]$, we can finish the proof.

\end{proof}

\begin{lemma}
\label{lemma:Lem-10}
There exists a constant $C_{12}$ depending only on $ \Omega,T,\lambda,\Lambda,R_\alpha,
\|f({\bf x},t)\|_{L^\infty(0,T,L^\infty(\Omega))}$ and
  $ \|{\bf U}\|_{L^\infty(0,T,L^\infty(\Omega))} $ such that
\begin{eqnarray}
\label{equation:eq-911}
\int_{Q_T}\frac{\kappa|\nabla \xi|^2}{(1+|\xi|^{\frac{3}{2}})}\leq C_{12}(\|{\bf B}_0\|_0+\kappa\|\nabla \theta_0\|_0^2)
\end{eqnarray}

\end{lemma}
\begin{proof}
Defining $ h(s)=sign(s)[1-(1+|s|)^{-\frac{1}{2}}]$, we have $|h(s)|\leq 1.$
For any $\eta\in \mathcal{Y}$, there holds $ h(\eta)\in \mathcal{Y}$ and
$\nabla h(\eta)=\frac{\nabla\eta}{2(1+|\eta|)^{\frac{3}{2}}}$.

Let $H(s)$ be the primitive function of $h(s)$ defined by
$$ H(s)=\int_0^s h(s')ds'=2+|s|-\frac{2}{1+|s|}\geq 0,$$
which implies
\begin{eqnarray}
\label{equation:eq-87}
\int_0^T\int_\Omega \frac{\partial \xi}{\partial t}h(\xi)=\int_0^T\frac{\partial}{\partial t}\int_\Omega H(\xi)=\int_\Omega H(\xi(T))\geq 0.
\end{eqnarray}

Taking $\Upsilon=h(\Upsilon)$ in (\ref{equation:eq-38}), we have
\begin{eqnarray}
(\partial_t\xi, h(\Upsilon))&+&(\kappa\nabla\xi, \nabla h(\Upsilon))
+<(\Psi(\xi+\theta_0)-\Psi(\theta_0)),h(\Upsilon)>_{\Gamma_2}\nonumber\\
\label{equation:eq-88}
&=&
([q(\xi)\mathcal{K}({\bf B})]_\epsilon,h(\Upsilon))-(\kappa\nabla\theta_0,\nabla h(\Upsilon)),
\end{eqnarray}
From (\ref{equation:eq-87}),we have
\begin{eqnarray}
&&\int_{0}^T\int_\Omega\kappa \nabla\xi\cdot \nabla h(\Upsilon)
+\int_0^T\int_{\Gamma_2}(\Psi(\xi+\theta_0)-\Psi(\theta_0))h(\Upsilon)\nonumber\\
\label{equation:eq-89}
&\leq &\int_{0}^T\int_\Omega[q(\xi)\mathcal{K}({\bf B})]_\epsilon h(\Upsilon)
-\int_{0}^T\int_\Omega\kappa\nabla\theta_0\nabla h(\Upsilon).
\end{eqnarray}
Now we estimate (\ref{equation:eq-88}) term by term.
\begin{eqnarray}
\label{equation:eq-90}
&& \int_{0}^T\int_\Omega\kappa \nabla\xi\cdot \nabla h(\Upsilon)=\frac{1}{2}
\int_{0}^T\int_\Omega \frac{\kappa|\nabla \xi|^2}{(1+|\xi|^{\frac{3}{2}})},\\
\label{equation:eq-91}
&&(\Psi(\xi+\theta_0)-\Psi(\theta_0))h(\Upsilon)\geq [1-(1+|\xi|)^{-\frac{1}{2}}]\frac{\zeta}{8}(|\xi|^4+\omega|\xi|)\geq 0.
\end{eqnarray}
From (\ref{equation:eq-84}) and Lemma \ref{lemma:Lem-71},
we have
\begin{eqnarray}
\label{equation:eq-92}
&&\int_0^T\int_\Omega[q(\xi)\mathcal{K}({\bf B})]_\epsilon h(\Upsilon)\leq
\int_0^T\int_\Omega q(\xi)\mathcal{K}({\bf B})
\leq C_{13}\|{\bf B}_0\|_0,\\
\label{equation:eq-93}
&& \int_0^T\int_\Omega  \kappa\nabla\theta_0\nabla h(\Upsilon)
\leq \|\kappa\theta_0\|_{L^2(\Omega)}[ \int_{Q_T}\frac{|\nabla\xi|^2}{(1+|\xi|)^3}]^{\frac{1}{2}}.
\end{eqnarray}
By Young's inequality in (\ref{equation:eq-93}), we have
  \begin{eqnarray}
  \frac{1}{2}
\int_{0}^T\int_\Omega \frac{\kappa|\nabla \xi|^2}{(1+|\xi|^{\frac{3}{2}})}
\leq C_{13}\|{\bf B}_0\|_0+C_{14}\|\kappa\nabla\theta_0\|_{L^2(\Omega)},
\end{eqnarray}
which implies the estimate of this lemma.
\end{proof}

\begin{lemma}
\label{lemma:Lem-11}
Assume that $1\leq q\leq \frac{5}{4}$, there exists a constant $C>0$ such that
\begin{eqnarray}
\label{equation:eq-95}
\|\xi\|_{L^{\frac{4q}{3}(Q_T)}}+\|\nabla\xi\|_{L^q(Q_T)}\leq C.
\end{eqnarray}
\end{lemma}

\begin{proof}
Taking $p=\frac{4q}{3}, q_1=\frac{3q}{3-q}$, by the Cauchy-Schwarz inequality, we have
\begin{eqnarray}
\label{equation:eq-96}
\int_\Omega|\xi(t)|^p&=&\int_\Omega|\xi(t)|^{\frac{q}{3}}|\xi(t)|^q
\leq [\int_\Omega|\xi(t)|^{\frac{q}{3}\cdot \frac{3}{q}}]^{\frac{q}{3}}
[\int_\Omega|\xi(t)|^{q\cdot \frac{3}{3-q}}]^{1-\frac{q}{3}}\nonumber\\
&=&\|\xi(t)\|_{L^1(\Omega)}^{\frac{q}{3}}[\int_\Omega|\xi(t)|^{q_1}]^{1-\frac{q}{3}}
\end{eqnarray}
By the embedding of $W^{1,q}\hookrightarrow L^{q_1}$ and using Poincare's inequality, we have
\begin{eqnarray}
\label{equation:eq-97}
\int_{Q_T}|\xi|_p\leq C\|\xi\|_{L^\infty(0,T;L^1(\Omega))}^{\frac{q}{3}}\|\nabla\xi\|_{L^q(Q_T)}^q
\leq C\|\nabla\xi\|_{L^q(Q_T)}^q.
\end{eqnarray}

Taking $ r=\frac{5-4q}{3}$, then we have $ p=\frac{(1+r)q}{2-q}$. By lemma 10 and Cauchy-Schwarz inequality, we have
\begin{eqnarray}
&&\int_{Q_T}|\nabla\xi|^q=\int_{Q_T}
\frac{|\nabla\xi|^q}{(1+|\xi|)^{\frac{q(1+r)}{2}}}{(1+|\xi|)^{\frac{q(1+r)}{2}}}\nonumber\\
&&\leq [\int_{Q_T}\frac{|\nabla\xi|^q}{(1+|\xi|)^{\frac{q(1+r)}{2}}}    ]^{\frac{q}{2}}[\int_{Q_T} {(1+|\xi|)^{\frac{q(1+r)}{2-q}}}]^{1-\frac{q}{2}}\nonumber\\
\label{equation:eq-98}
&&\leq [1+\|\xi \|^p_{L^p(Q_T)}]\leq C(1+\|\nabla\xi\|_{L^q(Q_T)}^{q(1-\frac{q}{2})}),
\end{eqnarray}
which implies $ \|\nabla\xi\|_{L^q(Q_T)}\leq C$, and (\ref{equation:eq-97})
implies $\|\xi\|_{L^p(Q_T)}\leq C.$

\end{proof}

\section{Well-posedness of the Source Problem}
We will prove the well-posedness of the problem (\ref{equation:eq-19})-
(\ref{equation:eq-20}), that is, we will investigate the limit of the solution of
the regularized problem (\ref{equation:eq-37})-
(\ref{equation:eq-38}) as the regularization parameter $\epsilon\rightarrow 0$.
For convenience, we denote the solutions of (\ref{equation:eq-37})-
(\ref{equation:eq-38}) by $({\bf B}_\epsilon,\xi_\epsilon).$
Then the regularized problem (\ref{equation:eq-37})-
(\ref{equation:eq-38})  can be represented by:find
 ${\bf B}_\epsilon\in L^2(0,T;{\mathcal{V}}) $ and $\xi_\epsilon \in L^2(0,T;\mathcal{Y})$
 such that
\begin{eqnarray}
(\partial_t{ \bf B}_\epsilon,{\Phi})&+&(\lambda(\xi+\theta_0)\nabla\times {\bf B}_\epsilon,\nabla\times{\Phi})
+{\Lambda}(\nabla\cdot{\bf B}_\epsilon,\nabla\cdot{\Phi})=R_\alpha(\frac{f({\bf x},t){\bf B}_\epsilon}{1+\gamma|{\bf B}_\epsilon|^2}, \nabla\times{\Phi})\nonumber\\
\label{equation:eq-99}
&+&({\bf U}\times{\bf B}_\epsilon,\nabla\times{\Phi}), \ \ \forall {\Phi\in \mathcal{V}},\\
(\partial_t\xi_\epsilon, \Upsilon)&+&(\kappa\nabla\xi_\epsilon, \nabla\Upsilon)
+<(\Psi(\xi_\epsilon+\theta_0)-\Psi(\theta_0)),\Upsilon>_{\Gamma_2}\nonumber\\
\label{equation:eq-100}
&=&
([q(\xi_\epsilon)\mathcal{K}({\bf B}_\epsilon)]_\epsilon,\Upsilon)-(\kappa\nabla\theta_0,\nabla\Upsilon),\forall \Upsilon\in
 \mathcal{Y}.
\end{eqnarray}

From Lemma \ref{lemma:Lem-71}, there exists ${\bf B}\in L^2(0,T;\mathcal{V})$ and a sequence ${\bf B}_\epsilon$ such that
$$ {\bf B}_\epsilon\rightarrow {\bf B}, \ in\  L^2(0,T;\mathcal{V}),\ \ \
{\bf B}_\epsilon\rightarrow {\bf B}, \nabla\times {\bf B}_\epsilon \rightarrow \nabla\times{\bf B},\ in\ L^1(Q_T).$$

From lemma \ref{lemma:Lem-11},there exists a $\xi\in W^{1,q}(Q_T)$ and a sequence $\xi_\epsilon$ such that
$$\xi_\epsilon\rightarrow \xi\ in W^{1,q}(Q_T),\forall q\in [1,\frac{5}{4}).  $$
Since $q(\xi)$ is bounded and Lipschitz continuous, we know that
$$q(\xi_\epsilon)\rightarrow q(\xi), \ \ a.e.\ \ in\ \ Q_T.$$

\begin{theorem}
Let ${\bf B}$ be the limit of the approximate solutions ${\bf B}_\epsilon$ as $\epsilon\rightarrow 0$. Then   ${\bf B}$  satisfies the weak formulation (\ref{equation:eq-19}) together with the initial condition ${\bf B}(0)={\bf B}_0(x)$.
\end{theorem}
\begin{proof}
The proof is parallel to that of Theorem 2 and we omit the details here.

\end{proof}

\begin{lemma}
\label{lemma:lem-12}
There exists a subsequence of ${\bf B}_\epsilon$ denoted still by the same notation such that
\begin{eqnarray}
\label{equation:eq-101}
\lim\limits_{\epsilon\rightarrow 0}\|q(\xi_\epsilon)\mathcal{K}({\bf B}_\epsilon)-
q(\xi)\mathcal{K}({\bf B})\|_{L^1(Q_T)}=0.
\end{eqnarray}
\end{lemma}

\begin{proof}

Firstly, we have
\begin{eqnarray}
\label{equation:eq-102}
&&\lim\limits_{\epsilon\rightarrow 0}\int_{Q_T}|q(\xi_\epsilon)\mathcal{K}({\bf B}_\epsilon)-
q(\xi)\mathcal{K}({\bf B})|\nonumber\\
&=&\lim\limits_{\epsilon\rightarrow 0}\int_{Q_T}
|q(\xi_\epsilon)(| \nabla\times {\bf B}_\epsilon|^2-| \nabla\times {\bf B}|^2)  |
+(q(\xi_\epsilon)-q(\xi))| \nabla\times {\bf B}|^2\nonumber\\
&=&\lim\limits_{\epsilon\rightarrow 0}\int_{Q_T}[q(\xi_\epsilon)(| \nabla\times {\bf B}_\epsilon|-| \nabla\times {\bf B}|)(| \nabla\times {\bf B}_\epsilon|+| \nabla\times {\bf B}|)\nonumber\\
&&\hspace{0.5cm}+(q(\xi_\epsilon)-q(\xi))| \nabla\times {\bf B}|^2]=0
\end{eqnarray}
 Secondly, we have
\begin{eqnarray}
\label{equation:eq-103}
&&\lim\limits_{\epsilon\rightarrow 0}\int_{Q_T}|q(\xi_\epsilon)\nabla\times{\bf B}_\epsilon\cdot ({\bf U}\times{\bf B}_\epsilon )  -q(\xi)\nabla\times{\bf B}\cdot ({\bf U}\times{\bf B} )|\nonumber\\
&&=\lim\limits_{\epsilon\rightarrow 0}\int_{Q_T}
(q(\xi_\epsilon)-q(\xi))\nabla\times{\bf B}_\epsilon\cdot ({\bf U}\times{\bf B}_\epsilon )+q(\xi)\nabla\times{\bf B}_\epsilon\cdot({\bf U}\times ({\bf B}_\epsilon-{\bf B}) )\nonumber\\
&&\hspace{0.5cm}+q(\xi)(\nabla\times{\bf B}_\epsilon-\nabla\times{\bf B})\cdot ({\bf U}\times{\bf B} )\nonumber\\
&&\leq\lim\limits_{\epsilon\rightarrow 0}\int_{Q_T}| (q(\xi_\epsilon)-q(\xi))\nabla\times{\bf B}_\epsilon\cdot ({\bf U}\times{\bf B}_\epsilon )|\nonumber\\
&&+\lim\limits_{\epsilon\rightarrow 0}\int_{Q_T}|q(\xi)(\nabla\times{\bf B}_\epsilon-\nabla\times{\bf B})\cdot ({\bf U}\times{\bf B} )|\nonumber\\
&&+\lim\limits_{\epsilon\rightarrow 0}\int_{Q_T}|q(\xi)(\nabla\times{\bf B}_\epsilon-\nabla\times{\bf B})\cdot ({\bf U}\times{\bf B} )|=0,
\end{eqnarray}
Thirdly, we have
\begin{eqnarray}
\label{equation:eq-104}
&&\lim\limits_{\epsilon\rightarrow 0}
R_\alpha\int_{Q_T}|q(\xi_\epsilon)\nabla\times{\bf B}_\epsilon\cdot\frac{f({\bf x},t){\bf B}_\epsilon}{1+\gamma|{\bf B}_\epsilon|^2}-q(\xi)\nabla\times{\bf B}\cdot\frac{f({\bf x},t){\bf B}}{1+\gamma|{\bf B}|^2}|\nonumber\\
&&\leq \lim\limits_{\epsilon\rightarrow 0}R_\alpha\int_{Q_T}
|(q(\xi_\epsilon)-q(\xi))\nabla\times{\bf B}_\epsilon\cdot\frac{f({\bf x},t){\bf B}_\epsilon}{1+\gamma|{\bf B}_\epsilon|^2}|\nonumber\\
&&+\lim\limits_{\epsilon\rightarrow 0}R_\alpha\int_{Q_T}|q(\xi)\nabla\times{\bf B}_\epsilon\cdot f({\bf x},t)\cdot (\frac{{\bf B}_\epsilon}{1+\gamma|{\bf B}_\epsilon|^2}-\frac{{\bf B}}{1+\gamma|{\bf B}|^2})|\nonumber\\
&&+\lim\limits_{\epsilon\rightarrow 0}R_\alpha\int_{Q_T}
|q(\xi)\frac{f({\bf x},t)}{1+\gamma|{\bf B}|^2}(\nabla\times{\bf B}_\epsilon-\nabla\times{\bf B})|=0.
\end{eqnarray}
From (\ref{equation:eq-102})-(\ref{equation:eq-104}), by using the triangle inequality, (\ref{equation:eq-101}) can be proved.

\begin{theorem}
Let $\xi$ be the limit of the approximate solutions $\xi_\epsilon$ as $\epsilon \rightarrow 0.$  Then $\xi $ satisfies the
weak formulation (\ref{equation:eq-20}) together with the initial condition $\xi(0)=0.$
\end{theorem}

\begin{proof}
Define the function: for any $>0$
$$g_r(s)=\frac{1}{1+rs^4}, G_r(s)=\int_0^sg_r(s')ds'.$$
It is easy to see that $ G_r$ is  a primitive function of $g_r$ and it satisfies  $|g_r(s)|\leq 1, |G_r|\leq T.$  Since $\xi_\epsilon\rightarrow \xi$ in $W^{1,6/5}(Q_T)$ when $\epsilon\rightarrow 0$, it is easy to see $ g_r(\xi_\epsilon)\rightarrow g_r(\xi)$ in $W^{1,6/5}(Q_T)$ and $G_r(\xi_\epsilon)\rightarrow G_r(\xi) $ in $W^{1,6/5}(Q_T)$.
Moreover, $g_r(\xi_\epsilon)$ and $G_r(\xi_\epsilon)$
are uniformly bounded with respect to
$\epsilon$, we infer that
$$G_r(\xi_\epsilon)\rightarrow G_r(\xi), g_r(\xi_\epsilon)\rightarrow g_r(\xi) \ \ \ in W^{1,6/5}(Q_T).$$

For any $v\in C_0^\infty(0,T;C^\infty(\Omega))$, let $ \Upsilon_\epsilon=vg_r(\xi_\epsilon),\ \ \Upsilon=vg_r(\xi).$
Clearly, we have $\phi_\epsilon\rightarrow \phi$ in $Q_T$. The proof consists of two steps.

From (\ref{equation:eq-100}), we have
\begin{eqnarray}
(\partial_t\xi_\epsilon, \Upsilon_\epsilon)&+&(\kappa\nabla\xi_\epsilon, \nabla\Upsilon_\epsilon)
+<(\Psi(\xi_\epsilon+\theta_0)-\Psi(\theta_0)),\Upsilon_\epsilon>_{\Gamma_2}\nonumber\\
\label{equation:eq-105}
&=&
([q(\xi_\epsilon)\mathcal{K}({\bf B}_\epsilon)]_\epsilon,\Upsilon_\epsilon)-(\kappa\nabla\theta_0,\nabla\Upsilon_\epsilon),
\end{eqnarray}
It is easy to see that
\begin{eqnarray}
&&\lim\limits_{\epsilon\rightarrow 0}\int_{Q_T}\partial_t\xi_\epsilon\Upsilon_\epsilon
=\lim\limits_{\epsilon\rightarrow 0}\int_{Q_T}\frac{\partial G_r(\xi_\epsilon)}{\partial t}v\nonumber\\
\label{equation:eq-106}
&&=-\lim\limits_{\epsilon\rightarrow 0}\int_{Q_T}
G_r(\xi_\epsilon)\frac{\partial v}{\partial t}
=\int_{Q_T}
G_r(\xi)\frac{\partial v}{\partial t}
=\int_{Q_T}\frac{\partial \xi}{\partial t}.
\end{eqnarray}
At the same time, since $ \xi_\epsilon\rightarrow \xi$ in $ W^{1,q}(Q_T),\forall q\in[1,\frac{5}{4})$, there holds
\begin{eqnarray}
&&\lim\limits_{\epsilon\rightarrow 0}\int_{Q_T}\kappa\nabla\xi_\epsilon \cdot \nabla\Upsilon_\epsilon\nonumber\\
&=&\lim\limits_{\epsilon\rightarrow 0}\int_{Q_T}
\kappa g_r(\xi_\epsilon)\nabla\xi_\epsilon \cdot \nabla  v-
\lim\limits_{\epsilon\rightarrow 0}\int_{Q_T}\frac{4r\kappa \xi_\epsilon^3}{1+r\xi_\epsilon^4}|\nabla\xi_\epsilon|^2\nonumber\\
&=&\int_{Q_T}\kappa g_r(\xi)\nabla\xi \cdot \nabla  v
-\int_{Q_T}\frac{4r\kappa \xi^3}{1+r\xi^4}|\nabla\xi|^2\nonumber\\
\label{equation:eq-107}
&=&\int_{Q_T}\kappa\nabla\xi \cdot \nabla\Upsilon.
\end{eqnarray}

From Lemma \ref{lemma:lem-8}
there exists a subsequence denoted by the same notation such that
$ \xi_\epsilon\rightarrow \xi$ in $L^3(0,T;\Gamma_2)$. This implies that
$$ (\zeta|\xi_\epsilon+\theta_0|^3+\omega)g_r(\xi_\epsilon)\rightarrow
(\zeta|\xi+\theta_0|^3+\omega)g_r(\xi), \ \ a.e.\  in \ (0,T)\times \Gamma_2. $$
The third term of  (\ref{equation:eq-105}) satisfies
\begin{eqnarray}
\label{equation:eq-108}
&&\lim\limits_{\epsilon\rightarrow 0}\int_{0}^T\int_{\Gamma_2}\Psi(\xi_\epsilon+\theta_0)\Upsilon_\epsilon=
\int_{0}^T\int_{\Gamma_2}\Psi(\xi+\theta_0)\Upsilon.
\end{eqnarray}

For the righthand side of (\ref{equation:eq-105}), by lemma \ref{lemma:lem-12}, we have
\begin{eqnarray}
&&\lim\limits_{\epsilon\rightarrow 0}[\int_{Q_T}
[q(\xi_\epsilon)\mathcal{K}({\bf B}_\epsilon)]_\epsilon\Upsilon_\epsilon+\int_0^T\int_{\Gamma_2}\Psi(\theta_0)
\Upsilon_\epsilon-
\int_{Q_T}\kappa\nabla\theta_0\nabla\Upsilon_\epsilon]\nonumber\\
\label{equation:eq-109}
&&=\int_{Q_T}
q(\xi)\mathcal{K}({\bf B})\Upsilon+\int_0^T\int_{\Gamma_2}\Psi(\theta_0)\Upsilon-
\int_{Q_T}\kappa\nabla\theta_0\nabla\Upsilon
\end{eqnarray}
From (\ref{equation:eq-106})-(\ref{equation:eq-109}) and (\ref{equation:eq-105}), we can get
\begin{eqnarray}
&&\int_{Q_T}\frac{\partial\xi}{\partial t}\Upsilon+
\int_{Q_T}\kappa\nabla\xi\cdot \nabla\Upsilon
+\int_0^T\int_{\Gamma_2}\Psi(\xi+\theta_0)\Upsilon\nonumber\\
\label{equation:eq-110}
&&=\int_{Q_T}
q(\xi)\mathcal{K}({\bf B})\Upsilon+\int_0^T\int_{\Gamma_2}\Psi(\theta_0)\Upsilon-
\int_{Q_T}\kappa\nabla\theta_0\nabla\Upsilon.
\end{eqnarray}
The initial condition $\xi(0)=0 $ can be proved similarly as in the proof of Theorem 3. We
do not elaborate on the details here.

For the function $g_r(\xi)$, we know $ g_r(\xi)\rightarrow 1, \ a.e. in \ Q_T$. Then we have
$$ \lim\limits_{r\rightarrow 0}\int_{Q_T}\kappa\nabla\xi\cdot \nabla g_r(\xi)= \lim\limits_{r\rightarrow 0}\int_{Q_T}
-\frac{4r\kappa\xi^3}{1+r\xi^4}|\nabla\xi|^2=0.$$
We can get
$$\lim\limits_{r\rightarrow 0}\int_{Q_T}[\frac{\partial \xi}{\partial t}\Upsilon+\kappa\nabla\xi\cdot\nabla\Upsilon]
=\int_{Q_T}[\frac{\partial \xi}{\partial t} v+\kappa\nabla\xi\cdot\nabla v]. $$
Since $\|\xi_\epsilon\|_{L^4(0,T;\Gamma_2)}\leq C$,  there exists a subsequence such that $ \|\xi\|_{L^4(0,T;\Gamma_2)}\leq \lim\limits_{\epsilon\rightarrow 0}\|\xi_\epsilon\|_{L^4(0,T;\Gamma_2)}\leq C.$
Then we have
$$ \lim\limits_{r\rightarrow 0}\int_{0}^T\int_{\Gamma_2}\Psi(\xi+\theta_0)\Upsilon
=\int_{0}^T\int_{\Gamma_2}\Psi(\xi+\theta_0)v.$$

Similarly, the righthand side converges to the form
$$ \int_{Q_T}
q(\xi)\mathcal{K}({\bf B})v +\int_0^T\int_{\Gamma_2}\Psi(\theta_0)v-
\int_{Q_T}\kappa\nabla\theta_0\nabla v.$$

Collecting all the above equalities and using (\ref{equation:eq-110}), we finally get:for $\forall v\in C_0^\infty(0,T;\mathcal{Y} \cap C^\infty(\Omega))$
\begin{eqnarray}
&&\int_{Q_T}\frac{\partial\xi}{\partial t}v+
\int_{Q_T}\kappa\nabla\xi\cdot \nabla v
+\int_0^T\int_{\Gamma_2}\Psi(\xi+\theta_0)v\nonumber\\
\label{equation:eq-111}
&&=\int_{Q_T}
q(\xi)\mathcal{K}({\bf B})v+\int_0^T\int_{\Gamma_2}\Psi(\theta_0)v-
\int_{Q_T}\kappa\nabla\theta_0\nabla v \ \ .
\end{eqnarray}
By the arbitrariness of v, we conclude (\ref{equation:eq-20}).
\end{proof}

In the foloowing, we present the stability of the solutions of the problem
(\ref{equation:eq-19})-(\ref{equation:eq-20}) from Lemma 8-Lemma 11 directly.
\begin{theorem}
Let ${\bf B}, \xi$ be the limits of  ${\bf B}_\epsilon, \xi_\epsilon $ given by (\ref{equation:eq-99}) and (\ref{equation:eq-100}), respectively. Then $ ({\bf B},\xi)$
solves the weak problem (\ref{equation:eq-19})-(\ref{equation:eq-20}). Furthermore,
\begin{eqnarray}
\label{equation:eq-112}
\|{\bf B}\|_{L^2(0,T;\mathcal{V})}+\|\xi\|_{L^{\frac{4q}{3}}(\Omega)}
+\|\nabla\xi\|_{L^q(Q_T)}\leq C, \ \ \forall q\in [1, \frac{5}{4}),
\end{eqnarray}
where $C$ depending on $ \Omega,T,\lambda,\Lambda,R_\alpha,
\|f({\bf x},t)\|_{L^\infty(0,T,L^\infty(\Omega))}$ and
  $ \|{\bf U}\|_{L^\infty(0,T,L^\infty(\Omega))}$.
\end{theorem}

\end{proof}


In the end, we give the uniqueness analysis of  the solutions of the problem (\ref{equation:eq-19})-(\ref{equation:eq-20}).
\begin{theorem}
Assume that ${\bf B}\in L^\infty (0,T;W^{1,4}({curl},\Omega))$, $\xi\in L^2(0,T;H^1(\Omega))\cap L^\infty (0,T;L^\infty(\Omega))$, $\lambda, \sigma$ satisfy the Lipsctiz continuous, $U, f \in L^\infty(0,T;L^\infty(\Omega))$, then the equations (\ref{equation:eq-19})-(\ref{equation:eq-20}) have a unique solution pair (${\bf B}, \xi$).
\end{theorem}
\begin{proof}
Assume $({\bf B}_1, \xi_1)$ and $({\bf B}_2, \xi_2)$ are two solutions of (\ref{equation:eq-19})-(\ref{equation:eq-20}), with ${\bf B}_i$ stays bounded in $L^\infty (0,T;W^{1,4}({curl},\Omega))$ and $\xi_i$ stays bounded in $L^2(0,T;H^1(\Omega))\cap L^\infty (0,T;L^\infty(\Omega))$, for $i=1, 2$. By denoting $\tilde{\bf B}={\bf B}_1-{\bf B}_2, \tilde{\xi}=\xi_1-\xi_2$ and setting $\Phi=\tilde{\bf B}, \Upsilon=\tilde{\xi}$, we get
\begin{eqnarray}
\label{equation:uni-001}
&&\frac{1}{2}\frac{d}{dt}\|\tilde{\bf B}\|_0^2+(\lambda(\theta_1)\nabla\times\tilde{\bf B},\nabla\times\tilde{\bf B})+((\lambda(\theta_1)-\lambda(\theta_2))\nabla\times{\bf B}_2,\nabla\times\tilde{\bf B})\nonumber\\
&&=R_\alpha(\frac{f({\bf x},t){\bf B}_1}{1+\gamma|{\bf B}_1|^2}-\frac{f({\bf x},t){\bf B}_2}{1+\gamma|{\bf B}_2|^2},\nabla\times\tilde{\bf B})+({\bf U}\times\tilde{\bf B},\nabla\times\tilde{\bf B}),\\
\label{equation:uni-002}
&&\frac{1}{2}\frac{d}{dt}\|\tilde{\xi}\|_0^2-(\kappa\nabla\tilde{\xi},\nabla\tilde{\xi})+\langle\Psi(\xi_1+\theta_0)
-\Psi(\xi_2+\theta_0),\tilde{\xi}\rangle_{\Gamma_2}\nonumber\\
&&=([q(\xi_1)K({\bf B}_1)]_\epsilon-[q(\xi_2)K({\bf B}_2)]_\epsilon,\tilde{\xi}).
\end{eqnarray}

For the first error equation, based on the equality that $\lambda(\theta_1)-\lambda(\theta_2) = \lambda' (\eta) \tilde{\xi}$, with $\eta$ between $\theta_1$ and $\theta_2$, we have
\begin{eqnarray}
  &&
  - ((\lambda(\theta_1)-\lambda(\theta_2))\nabla\times{\bf B}_2,\nabla\times\tilde{\bf B})
  = - ( \lambda' (\eta) \tilde{\xi} \, \nabla\times{\bf B}_2,\nabla\times\tilde{\bf B}) \nonumber
\\
  &\le&
  \| \lambda' (\eta) \|_{L^\infty(\Omega)} \| \tilde{\xi} \|_{L^4 (\Omega)} \| \nabla\times{\bf B}_2 \|_{L^4 (\Omega)} \| \nabla\times\tilde{\bf B} \|_0  \nonumber
\\
  &\le&
  C \| \tilde{\xi} \|_{L^4 (\Omega)} \| \nabla\times{\bf B}_2 \|_{L^4 (\Omega)} \| \nabla\times\tilde{\bf B} \|_0 \le C \| \tilde{\xi} \|_{L^4 (\Omega)} \| \nabla\times\tilde{\bf B} \|_0 ,  \label{uniqueness-1}
\end{eqnarray}
in which the last two steps come from the fact that both $\theta_1$ and $\theta_2$ stay bounded in $L^\infty (0,T; L^\infty (\Omega))$, and ${\bf B}_2$ stays bounded in $L^\infty (0,T; W^{1,4}({curl},\Omega))$. The right hand side of (\ref{equation:uni-001}) could be bounded in a more straightforward way:
\begin{eqnarray}
  &&
  R_\alpha(\frac{f({\bf x},t){\bf B}_1}{1+\gamma|{\bf B}_1|^2}-\frac{f({\bf x},t){\bf B}_2}{1+\gamma|{\bf B}_2|^2},\nabla\times\tilde{\bf B})
  \le 2 R_\alpha \|f\|_{L^\infty(\Omega)} \|\tilde{\bf B}\|_{0} \|\nabla\times\tilde{\bf B}\|_0 ,
  \label{uniqueness-2}
\\
  && \mbox{since} \, \,
  \left| \frac{{\bf B}_1}{1+\gamma|{\bf B}_1|^2}-\frac{{\bf B}_2}{1+\gamma|{\bf B}_2|^2}  \right|
  \le 2 | {\bf B}_1 - {\bf B}_2 | ,  \nonumber
\\
  &&
  ({\bf U}\times\tilde{\bf B},\nabla\times\tilde{\bf B}) \le  \|{\bf U}\|_{L^\infty(\Omega)}\|\tilde{\bf B}\|_{0}\|\nabla\times\tilde{\bf B}\|_{0} .    \label{uniqueness-3}
\end{eqnarray}
Therefore, a substitution of (\ref{uniqueness-1})-(\ref{uniqueness-3}) into (\ref{equation:uni-001}) yields
\begin{eqnarray}
\label{equation:uni-003}
\frac{1}{2}\frac{d}{dt}\|\tilde{\bf B}\|_0^2 + \lambda_{min} \|\nabla\times\tilde{\bf B}\|_0^2 \leq C ( \|\tilde{\xi}\|_{L^4(\Omega)} + \|\tilde{\bf B}\|_{0} ) \| \nabla \times \tilde{\bf B}\|_{0} .
\end{eqnarray}
For the second error equation (\ref{equation:uni-002}), a direct calculation shows that
\begin{eqnarray}
&&\frac{1}{2}\frac{d}{dt}\|\tilde{\xi}\|_0^2+\kappa_{min}\|\nabla\tilde{\xi}\|_0^2+\langle\Psi(\xi_1+\theta_0)
-\Psi(\xi_2+\theta_0),\tilde{\xi}\rangle_{\Gamma_2}\nonumber\\
&&\leq \left( (q(\xi_1)-q(\xi_2))  \mathcal{K}({\bf B}_2) , \tilde{\xi} \right) \nonumber\\
&&+ \Bigl( q(\xi_1)(|\nabla\times{\bf B}_1|^2-|\nabla\times{\bf B}_2|^2-(\nabla\times{\bf B}_1\cdot({\bf U}\times{\bf B}_1)-\nabla\times{\bf B}_2\cdot({\bf U}\times{\bf B}_2))\nonumber\\
&&-(R_\alpha\nabla\times{\bf B}_1\frac{f{\bf B}_1}{1+\gamma|{{\bf B}_1}|^2}-R_\alpha\nabla\times{\bf B}_2\frac{f{\bf B}_2}{1+\gamma|{{\bf B}_2}|^2}) ) ,  \tilde{\xi} \Bigr) .  \label{uniqueness-4}
\end{eqnarray}
The assumption that ${\bf B}_2$ stays bounded in $L^\infty (0,T; W^{1,4}({curl},\Omega))$ implies that
\begin{eqnarray}
  \| \mathcal{K}({\bf B}_2) \|_{L^\infty (0,T; L^2 (\Omega))} \le C . \label{uniqueness-5}
\end{eqnarray}
This in turn indicates that
\begin{eqnarray}
  &&
  \left( (q(\xi_1)-q(\xi_2))  \mathcal{K}({\bf B}_2) , \tilde{\xi} \right)
  = \left( q' (\eta) \xi  \mathcal{K}({\bf B}_2) , \tilde{\xi} \right)   \nonumber
\\
  &\le& C  \| \xi \|_{L^4 (\Omega)} \| \mathcal{K}({\bf B}_2) \|_{L^2 (\Omega)} \| \tilde{\xi} \|_{L^4 (\Omega)}
  \le C   \| \xi \|_{L^4 (\Omega)}^2 .  \label{uniqueness-6}
\end{eqnarray}
Again, the fact that both $\xi_1$ and $\xi_2$ stay bounded in $L^\infty (0,T; L^\infty (\Omega))$ has been used in the derivation. For the second expansion term on the right hand side of (\ref{uniqueness-4}), we see that
\begin{eqnarray}
  &&
  \left( q(\xi_1)(|\nabla\times{\bf B}_1|^2-|\nabla\times{\bf B}_2|^2 ) ,  \tilde{\xi} \right)
  = \left( q(\xi_1) ( \nabla\times ( {\bf B}_1 + {\bf B}_2 ) ) \cdot ( \nabla \times \tilde{\bf B} ) ,  \tilde{\xi} \right)  \nonumber
\\
  &\le&
  C ( \| \nabla\times {\bf B}_1 \|_{L^4 (\Omega)} + \| \nabla \times {\bf B}_2 \|_{L^4 (\Omega)} )  \| \nabla \times \tilde{\bf B} \|_0  \| \tilde{\xi} \|_{L^4 (\Omega)}  \nonumber
\\
  &\le&
   C \|\tilde{\xi}\|_{L^4(\Omega)} \| \nabla \times \tilde{\bf B}\|_0  .   \label{uniqueness-7}
\end{eqnarray}
The other terms on the right hand side of  (\ref{uniqueness-4}) could be analyzed in a similar way:
\begin{eqnarray}
  &&
  - \Bigl( q(\xi_1) (\nabla\times{\bf B}_1\cdot({\bf U}\times{\bf B}_1)-\nabla\times{\bf B}_2\cdot({\bf U}\times{\bf B}_2)) , \tilde{\xi} \Bigr) \nonumber
\\
  &\le&
  C \|\tilde{\xi}\|_{L^4(\Omega)} ( \| \tilde{\bf B} \|_0 + \| \nabla \times \tilde{\bf B}\|_0 ) ,    \label{uniqueness-8}
\\
  &&
  - R_\alpha \Bigl( q(\xi_1) ( \nabla\times{\bf B}_1\frac{f{\bf B}_1}{1+\gamma|{{\bf B}_1}|^2}-\nabla\times{\bf B}_2\frac{f{\bf B}_2}{1+\gamma|{{\bf B}_2}|^2}) ) ,  \tilde{\xi} \Bigr) \nonumber
\\
  &\le&
  C \|\tilde{\xi}\|_0 \| \tilde{\bf B} \|_0 .    \label{uniqueness-9}
\end{eqnarray}
And also, the estimate for the boundary integral term on the left hand side of \ref{uniqueness-4}) is trivial:
\begin{eqnarray}
  \langle\Psi(\xi_1+\theta_0)
-\Psi(\xi_2+\theta_0),\tilde{\xi}\rangle_{\Gamma_2} \ge 0 . \label{uniqueness-10}
\end{eqnarray}
Subsequently, a substitution of (\ref{uniqueness-6})-(\ref{uniqueness-10}) into (\ref{uniqueness-4}) results in
\begin{eqnarray}
  \frac{1}{2}\frac{d}{dt}\|\tilde{\xi}\|_0^2+\kappa_{min}\|\nabla\tilde{\xi}\|_0^2
\le C \|\tilde{\xi}\|_{L^4(\Omega)} ( \| \tilde{\bf B} \|_0 + \| \nabla \times \tilde{\bf B}\|_0 )
+ C \|\tilde{\xi}\|_{L^4(\Omega)}^2  .
 \label{equation:uni-004}
\end{eqnarray}

As a result, a combination of (\ref{equation:uni-003}) and (\ref{equation:uni-004}) yields
\begin{eqnarray}
  &&
  \frac{1}{2}\frac{d}{dt} ( \|\tilde{\bf B}\|_0^2 + \|\tilde{\xi}\|_0^2 ) + \lambda_{min} \|\nabla\times\tilde{\bf B}\|_0^2 +\kappa_{min}\|\nabla\tilde{\xi}\|_0^2 \nonumber
\\
  &\le&
   C_1 \|\tilde{\bf B}\|_{0} \| \nabla \times \tilde{\bf B}\|_{0}
   + C_2 \|\tilde{\xi}\|_{L^4(\Omega)} ( \| \tilde{\bf B} \|_0 + \| \nabla \times \tilde{\bf B}\|_0 )
  + C_3 \|\tilde{\xi}\|_{L^4(\Omega)}^2  . \label{equation:uni-005}
\end{eqnarray}
Furthermore, the following Sobolev inequality (in 3-D) is applied:
\begin{eqnarray}
  \| \tilde{\xi} \|_{L^4} \le C \| \tilde{\xi} \|_{H^{\frac34}} \le C \| \tilde{\xi} \|_0^{\frac14} \cdot \| \tilde{\xi} \|_1^{\frac34} \le C ( \| \tilde{\xi} \|_0 + \| \tilde{\xi} \|_0^{\frac14} \cdot \| \nabla\tilde{\xi} \|_0^{\frac34} , \label{equation:uni-006}
\end{eqnarray}
so that the following estimates become available:
\begin{eqnarray}
  &&
  C_1 \|\tilde{\bf B}\|_{0} \| \nabla \times \tilde{\bf B}\|_{0}
  \le  \frac{C_1^2}{\kappa_{min}} \|\tilde{\bf B}\|_{0}^2 + \frac14 \lambda_{min} \| \nabla \times \tilde{\bf B}\|_0^2 ,   \label{equation:uni-007-1}
\\
  &&
  C_2 \|\tilde{\xi}\|_{L^4(\Omega)} \| \tilde{\bf B} \|_0
  \le  C_4 ( \| \tilde{\xi} \|_0 + \| \tilde{\xi} \|_0^{\frac14} \cdot \| \nabla\tilde{\xi} \|_0^{\frac34} ) \| \tilde{\bf B} \|_0  \nonumber
\\
  &\le&
  C_5 ( \| \tilde{\xi} \|_0^2 + \| \tilde{\bf B} \|_0^2 ) + \frac14 \kappa_{min} \| \nabla \tilde{\xi} \|_0^2 ,   \label{equation:uni-007-2}
\\
  &&
  C_2 \|\tilde{\xi}\|_{L^4(\Omega)} \| \nabla \times \tilde{\bf B} \|_0
  \le  C_4 ( \| \tilde{\xi} \|_0 + \| \tilde{\xi} \|_0^{\frac14} \cdot \| \nabla\tilde{\xi} \|_0^{\frac34} ) \| \nabla \times \tilde{\bf B} \|_0  \nonumber
\\
  &\le&
  C_6 ( \| \tilde{\xi} \|_0^2 + \| \tilde{\bf B} \|_0^2 ) + \frac14 \lambda_{min} \| \nabla \times \tilde{\bf B}\|_0^2 + \frac14 \kappa_{min} \| \nabla \tilde{\xi} \|_0^2 ,   \label{equation:uni-007-3}
\\
  &&
   C_3 \|\tilde{\xi}\|_{L^4(\Omega)}^2 \| \nabla \times \tilde{\bf B} \|_0
  \le  C_7 ( \| \tilde{\xi} \|_0 + \| \tilde{\xi} \|_0^{\frac14} \cdot \| \nabla\tilde{\xi} \|_0^{\frac34} )^2  \nonumber
\\
  &\le&
  C_8 \| \tilde{\xi} \|_0^2 + \frac14 \kappa_{min} \| \nabla \tilde{\xi} \|_0^2 ,   \label{equation:uni-007-4}
\end{eqnarray}
in which Young's inequality has been extensively applied. Going back to (\ref{equation:uni-005}), we arrive at
\begin{eqnarray}
  &&
  \frac{1}{2}\frac{d}{dt} ( \|\tilde{\bf B}\|_0^2 + \|\tilde{\xi}\|_0^2 ) + \frac12 \lambda_{min} \|\nabla\times\tilde{\bf B}\|_0^2 + \frac14 \kappa_{min}\|\nabla\tilde{\xi}\|_0^2 \nonumber
\\
  &\le&
  (  \frac{C_1^2}{\kappa_{min}}  + C_5 + C_6 ) \| \tilde{\bf B} \|_0^2
  + (C_5 + C_6 + C_8) \| \tilde{\xi} \|_0^2 . \label{equation:uni-008}
\end{eqnarray}
Consequently, with an application of Gronwall inequality, and making use of the fact that $ \| \tilde{\bf B} ( \cdot, t=0) \|_0 =0$, $ \| \tilde{\xi} ( \cdot, t=0) \|_0 =0$, we arrive at
\begin{eqnarray}
  \| \tilde{\bf B} ( \cdot, t) \|_0 = 0 , \, \, \, \| \tilde{\xi} ( \cdot, t) \|_0 = 0 ,  \quad \forall t > 0. \label{equation:uni-009}
\end{eqnarray}
This completes the uniqueness proof.
\end{proof}

\section{Acknowledgements}
The first author is supported by
P.R. China NSFC (NO. 11471296, 11571389, 11101384).
 The second author is supported by Hong Kong Research Council GRF Grants  B-Q40W and 8-ZDA2.
The third author is supported by NSFC DMS-1418689 and NSFC 11271281.


\begin{thebibliography}{}
  \bibitem{MR01} Molokov,S., Moreau, R., Moffatt, H. K., :, Magnetohydrodynamics,  Springer,  Netherlands.(2007)

 \bibitem{MR02}Parker,E. N.,: Cosmical Magnetic Fields, Clarendon Press, Oxford, 1979.

 \bibitem{MR03}Cattaneo, F. and Hughes,D. W.,:  Nonlinear saturation of the turbulent alpha effect where a
large scale field is imposed, Phys. Rev. E, {\bf 54}, 4532-4535(1996)


 \bibitem{MR04} Moffatt,H. K.,: Magnetic Field Generation in Electrically Conducting Fluids, Cambridge
University Press, Cambridge, UK, 1978.


\bibitem{MR05}  Sanchez,S., Fournier,A., Pinheiro,K. J.  and Aubert,J.,: A mean-field Babcock-Leighton solar dynamo model with long-term variability, Anais da Academia Brasileira de Ci\^{e}ncias,  {\bf 86:1},11-26(2014)

\bibitem{MR06}
Brandenburg, A., and Subramanian,K.,: Astrophyiscal magnetic fields and nonlinear dynamo
theory, Physics Reports, {\bf 417},  1-209 (2005)




\bibitem{MR07} Yin, H. M.:, { Existence and regularity of
 a week solution to Maxwell's equations with a thermal effect},
  Math. Meth. Appl. Sci. {\bf 29}, 1199-1213 (2006)

 \bibitem{MR08}   Metaxas,A.C.,: Foundations of Electroheat,
 A Unified Approach, Wiley, New York, 1996.

  \bibitem{MR09} Elsayed, M.A. Elbashbeshy,  Emam,T.G., and Abdelgaber, K.M.,:
 {Effects of thermal radiation and magnetic field
on unsteady mixed convection flow and heat transfer over
an exponentially stretching surface with suction
in the presence of internal heat generation/absorption},
Journal of the Egyptian Mathematical Society, {\bf 20}, 215¨C222(2012)

 \bibitem{MR10} Ka\v cur, J.,:Method of {R}othe in evolution equations,
 Lecture Notes in Math.,Springer, Berlin,{\bf 1192},23--34(1986)

 \bibitem{MR11}Va\u\i nberg, M. M.,: Variational method and method of monotone operators in the
              theory of nonlinear equations, Halsted Press
              (A division of John Wiley \& Sons), New
              York-Toronto, Ont.; Israel Program for Scientific
              Translations, Jerusalem-London,1973

   \bibitem{MR12}Zeidler, E.,:Nonlinear functional analysis and its applications. {II}/{B}:
   Nonlinear monotone operators, Springer-Verlag, New York,1990

  \bibitem{MR13} Ranjit, N. K. and Shit, G. C., Joule heating effects on electromagnetohydrodynamic flow
              through a peristaltically induced micro-channel with different
              zeta potential and wall slip, Phys. A.,{\bf 482},458--476(2017)

    \bibitem{MR14}  Hossain, M. A. and Gorla, R. S. R.,Joule heating effect on magnetohydrodynamic mixed convection
              boundary layer flow with variable electrical conductivity,Internat. J. Numer. Methods Heat Fluid Flow, 23(2),275--288(2013)
\bibitem{MR15}
Berm\'udez, A. and Mu\~noz-Sola, R. and V\'azquez, R., Analysis of two stationary magnetohydrodynamics systems of
              equations including {J}oule heating, J. Math. Anal. Appl.,
              368, 444-468(2010)


    \bibitem{MR16}           Chovan, J. and Slodi\v cka, M.,Induction hardening of steel with restrained {J}oule heating
              and nonlinear law for magnetic induction field: solvability,
              J. Comput. Appl. Math., 311, 630--644(2017)

%
%
%
%
%
%



\end{thebibliography}

\end{document}